\documentclass{article}
\usepackage{amssymb, amsmath}
\usepackage{color}

\usepackage{amsfonts,amssymb,amsthm}
\usepackage{epsfig}
\usepackage{amsmath}
\usepackage{hyperref}
\usepackage{scrlfile}
\usepackage{srcltx}
\usepackage{graphics}
\usepackage{rotating}
\usepackage{remreset}
\newcommand{\beqa}{\begin{eqnarray}}
\newcommand{\eeqa}[1]{\label{#1}\end{eqnarray}}
\newcommand{\beq}{\begin{equation}}
\newcommand{\eeq}[1]{\label{#1}\end{equation}}

\def\R{\mathbb{R}}

\def\tg{{\tau_\gamma}}
\def\o{\omega}

\def\d{\delta}

\def\phi{\varphi}

\def\div{\mathop{\rm div}\nolimits}

\def\demifleche{\rightharpoonup}

\def\*fleche{\buildrel *\over\demifleche}

\def\tol2{\buildrel\hbox{$L^2$}\over\longrightarrow}

\def\toto{\leaders\hbox to 5mm{\hfil.\hfil}\hfill}

\def\div{\mathop{\rm div}\nolimits}

\newcommand{\SO}[1]{\operatorname{SO}(#1)}

\newcommand{\dist}{\operatorname{dist}}

\newcommand{\wtto}{\stackrel{2}\rightharpoonup}
\newcommand{\stto}{\stackrel{2}\rightarrow}
\newcommand{\twconv}{\stackrel{2}\rightharpoonup}

\newcommand{\abs}[1]{\left\vert #1 \right\vert}

\newtheorem{theorem}{Theorem}[section]
\newtheorem{proposition}[theorem]{Proposition}

\newtheorem{definition}[theorem]{Definition}
\newtheorem{lemma}[theorem]{Lemma}

\newtheorem{remark}[theorem]{Remark}

\newcommand{\e}{\varepsilon}

\begin{document}
\title{Bending of thin periodic plates 
}
\author{M. Cherdantsev\footnote{School of Mathematics, Cardiff University, Senghennydd Road, Cardiff, CF24 4AG, United Kingdom},\,\, K. D. Cherednichenko\footnote{Department of Mathematical Sciences, University of Bath, Claverton Down, Bath, BA2 7AY, United Kingdom}}

\maketitle

\begin{abstract}
We show that nonlinearly elastic plates of thickness $h\to 0$ with an $\varepsilon$-periodic structure 
such that $\varepsilon^{-2}h\to 0$ exhibit non-standard 
behaviour in the asymptotic two-dimensional reduction from three-dimensional elasticity: in general, their effective stored-energy density is ``discontinuously anisotropic'' 
in all directions. The proof relies on a new result concerning an additional isometric constraint that deformation fields must satisfy on the microscale.

\end{abstract}

\section{Introduction}

Understanding the behaviour of elastic plates (and, more generally, thin elastic structures) from the rigorous mathematical point of view 
has attracted much attention of applied analysts over the recent years. The related 
activity was initiated by the paper \cite{FJM2002} by Friesecke, James and M\"uller, concerning plate deformations $u$ with finite ``bending energy''. This work appeared alongside the thesis 
\cite{Pantz} by Pantz and was followed by a study of other energy scalings \cite{FJM_hierarchy}. It puts forward the idea that homogeneous plates of thickness $h$, viewed as 
three-dimensional nonlinearly elastic bodies, afford a special compactness argument for sequences of deformation gradients $\nabla u$ with elastic energy of order $O(h^3)$ as $h\to0.$ This argument is based on a new ``rigidity estimate'' \cite{FJM2002} concerning the distance of local values of $\nabla u$ from the group of rotations ${\rm SO}(3).$ The observation that this bounds the distance from a constant rotation field (which depends on $u$) allowed the authors of \cite{FJM2002} to show that the ``limit'' elastic energy functional as $h\to0$ is given by
\begin{equation}
E_{\rm lim}(u):=\frac{1}{12}\int\limits_{\omega} Q_2({\rm II}(u)), \,\, u \in H^2_{\rm iso}(\omega),
\label{FJMfunc}
\end{equation}
where $\omega\subset{\mathbb R}^2$ represents the mid-surface of the undeformed plate, ${\rm II}(u)= (\nabla' u)^\top\nabla'(\partial_1 u \wedge \partial_2 u),$ $\nabla'=(\partial_1,\partial_2),$ is the matrix of the second fundamental form of an isometric surface $u:\omega\to{\mathbb R}^3,$ and $Q_2$ is a quadratic form derived via 
dimension reduction from a quadratic form appearing in the process of linearisation of elastic properties of the material in the small-strain regime.

From the point of view of applications to real-world materials, it seems reasonable to ask in what way the above result is affected by a possible inhomogeneity of material properties of the 
plate in the directions tangential to its mid-surface. One can imagine, for instance, that the plate has a periodic structure of period $\varepsilon>0$ and try to replace the $\varepsilon, h$-dependent family of $h^{-3}$-scaled energy functionals by an ``effective'' functional in the sense of variational convergence \cite{DeGiorgiFranzoni}. Such analysis should reveal, in particular, 
whether the behaviour of the plate really depends on the relative orders of smallness of the parameters $\varepsilon$ and $h.$ Work on this programme was started in the thesis \cite{NeukammPhD2010} by Neukamm, where a series of similar questions was addressed for a periodically inhomogeneous rod, {\it {\it i.e.}} a one-dimensional analogue of a plate. Further, a recent paper \cite{HNV} has investigated the behaviour of periodic plates in the cases when $h\gg\varepsilon$ and $h\sim \e$. This was followed by the work \cite{Velcic}, where the case $\varepsilon^2\ll h\ll\varepsilon$ is addressed. In the present paper we develop an approach (see Sections \ref{compactness}, \ref{moderateregime}) that in our view simplifies the derivation of the corresponding $\Gamma$-limit, via a ``smoothing'' approximation procedure that precedes the two-scale compactness argument, see Section \ref{compactness} of the present work. Smoothing is known to be useful in the asymptotic 
analysis of sequences of solutions to parameter-dependent PDE (see {\it e.g.} \cite{Griso}, \cite{ZhikPast}). In the present work we exploit similar considerations in the asymptotic analysis of bounded-energy sequences, where the energy is represented by an integral functional. Our smoothing approach 
replaces the approximation result \cite[Lemma 3.8]{HNV}, when $h\gg\varepsilon$ or $h\sim\varepsilon,$ and an additional technical statement \cite[Lemma 3.7]{Velcic},
when $\varepsilon^2\ll h\ll\varepsilon,$ by a two-scale compactness theorem for the second gradients of smooth approximations (Theorem \ref{th2.1} below). Remarkably, in 
the situation where the deformation $u_h\in H^1(\Omega)$ has no second derivatives, the smoothing procedure takes us to a setting where the second derivatives exist and are bounded in the $L^2-$ norm, see Lemma \ref{l3.2}. 

The added value of our approach, however, is revealed through its ability to deal with the more difficult case $h\ll\varepsilon^2,$ which has remained open until now.   An additional key observation in handling this case is that the determinant of the matrix of second derivatives of the smoothed approximation can be rewritten in a form amenable to the use of a compensated compactness argument (see proof of Theorem \ref{isometryconstraint}) in order to derive an additional, "fast-scale", isometry constraint on the limit finite-energy deformations, see 
(\ref{constraint}). It also perhaps worth mentioning that the smoothing approach yields a somewhat shorter route to the statement of compactness of finite-energy sequences in the case of homogeneous plates considered in \cite{FJM2002}.
  

Notably, a recent paper \cite{NO} contains an analysis of ``zero-thickness'' plates, where functionals of the ``limit'' form (\ref{FJMfunc}) with explicit $\varepsilon$-periodic $x$-dependence of $Q_2$ are studied in the limit $\varepsilon\to0.$ Our results are consistent with those of \cite{NO}, in the sense of offering an alternative route to the same ``supercritically-thin plate'' limit functional. However, while the authors of  \cite{NO} work with a two-dimensional formulation from the outset, our derivation is in the spirit of ``dimension reduction'' from the full three-dimensional problem. In addition, our approach offers a shorter route to the proofs of \cite{NO}, which is checked directly.

All of the above earlier works follow a traditional methodology of $\Gamma$-convergence relying solely on convergence properties of sequences with bounded energy combined with "structure theorems" 
for combinations of their gradients, without analysing their asymptotic structure as a two-scale series in powers of the parameters $\varepsilon, h.$ Our approach, on the contrary, is guided by an analysis of the original sequence of energy functionals via two-scale power series expansions. (For a detailed discussion of the viewpoint provided by the method of asymptotic expansions, see Appendix B.) For ``supercritical'' scalings $h\ll\varepsilon^2$ it suggests in particular that, to the leading order, the approximating (``exact'' or ``recovery'') sequences have to satisfy an additional  constraint 
\beq
{\rm det}({\rm II}+\nabla^2_y \psi)=0,
\eeq{constraint}
where $\psi=\psi(x',y),$ $x'\in\omega,$ $y\in Y:=[0,1)^2,$ is a term that appears in the two-scale limit of the original sequence of deformations and is responsible for the behaviour of the plate on $\e$-scale. One key aspect that facilitates this approach is that the limit stored-energy function $Q_2$ is quadratic with respect to the second fundamental form of the surface, which makes the analysis of $\varepsilon, h$-dependent plate energies somewhat amenable to a perturbation technique.  

The observation that the space of deformations with finite bending energy acquires the constraint (\ref{constraint}), combined with a characterisation of isometric embeddings by Pakzad \cite{Pakzad} and the earlier ansatz of \cite{FJM2002}, see (\ref{5.120}), allows us to carry out a rigorous proof of variational convergence. Thanks to the two-scale asymptotic structure, the presence of the second gradient of the corrector $\psi$ in the two-scale limit of deformation sequences provides a natural construction for the recovery sequences, {\it cf.} (\ref{4.84}) for the moderate regime. The supercritical regime $h\ll\varepsilon^2$ is studied in Section \ref{extremeenergy}, which contains our main result, Theorem \ref{th5.6}: the limit energy functional is given by 
\[
E^{\rm sc}_{\rm hom}=\frac{1}{12}\int\limits_\omega Q^{\rm sc}_{\rm hom}\bigl({\rm II}(u)\bigr),
\]
where 
\[
Q^{\rm sc}_{\rm hom}=\min\int\limits_Y Q_2(y, {\rm II}+\nabla_y^2\psi)dy,\ \ \ \ \psi\in H^2(Y)\ {\rm periodic,\ subject\ to\ (\ref{constraint})}. 
\]
The isometry constraint (\ref{constraint}) implies, in particular, that $Q^{\rm sc}({\rm II})$ is ``discontinuously anisotropic in all directions of bending" as a function of the macroscopic deformation gradient 
$\nabla u,$ see Theorem \ref{th2.8} and the subsequent discussion. To the best of our knowledge, this is a new phenomenon for nonlinearly elastic plates. Finally, we remark that the analysis of the ``critical'' scaling $h\sim\varepsilon^2$ is currently open.

Throughout the text ${\mathcal I}_d$ denotes the  $d\times d$ identity matrix, $d=2,3.$ Unless indicated otherwise, we denote by $C$ a positive constant whose precise value is of no importance and may vary between different formulae. We use the notation $\partial_i$, $i=1,2,3$, for the partial derivative with respect to $x_i$, $\partial_{y_i}$, $i=1,2$, for the partial derivative with respect to $y_i$, and $\nabla_y$ for the gradient with respect to $y$.



\section{Setting of the problem}

Let $\omega$ be a bounded convex \footnote{The convexity of $\o$ is a technical requirement which allows us to use the results describing the properties of isometric immersions for the supercritical case $h\ll\e^2$. From the point of view of the aim of the present work this requirement is insignificant, since the limiting effective behaviour of an elastic plate is a local property.} domain in $\mathbb R^2$, and let $\Omega_h:= \omega\times(-\frac{h}{2}, \frac{h}{2}) $ be a reference configuration of an undeformed elastic plate, where $h$ is a small positive parameter, $0<h\ll 1$. A deformation $u$ of the plate $\Omega_h$ is an $H^1$-function from $\Omega_h$ into $\mathbb R^3$. We assume that the material properties of the plate $\Omega_h$ vary over $\omega$ with a period $\e$, which goes to zero simultaneously with $h$. For a plate made of a hyperelastic material (see {\it e.g.} \cite{Ciarlet}), the elastic energy of a deformation $u$ is given by the functional
\begin{equation*}
 \int\limits_{\Omega_h} W(\e^{-1}x', \nabla u) dx,
\end{equation*}
where $x':=(x_1,x_2) \in \omega$. The stored-energy density function $W(y, F)$ is assumed to be measurable and periodic with respect to $y\in Y:=[0,1)^2$ and to satisfy the standard conditions of nonlinear elasticity (see {\it e.g.} \cite{Ciarlet}):

A) The density $W$ is frame-indifferent, \textit{{\it i.e.}}  $W(y,RF)=W(F)$ for any $R\in
  {\rm SO}(3)$ and any $F\in\R^{3{\times}3}$, where ${\rm SO}(3)$ is a group of rotation matrices in $\R^{3{\times}3}$.

B) The identity deformation $u(x)=x$ is a natural state, \textit{{\it i.e.}} $W(y,{\mathcal I}_3)=\min_F W(y,F)=0$.

C) 
There exists a constant $C>0$ such that for all $y\in Y,$ $F\in\R^{3{\times}3}$ he inequality
  \begin{equation}\label{2.2}
    W(y,F) \geq C\,{\rm dist}^2(F,{\rm SO}(3))
  \end{equation}
holds.
  
Additionally we assume that ({\it cf.} \cite{CCN}, \cite{FJM2002}, \cite{HNV})

${\rm D}$) The density $W,$ as a function of the deformation gradient, admits a quadratic expansion at the identity, \textit{i.e.} there exists a non-negative quadratic form $Q_3$ on  $\R^{3{\times}3}$ such that
  \begin{equation}\label{WTaylor}
    W(y,{\mathcal I}_3+G) = Q_3(y,G) + r(y,G)
  \end{equation}
  for all $G\in\R^{3{\times}3},$ where 
\begin{equation*}
  r(y,G)=o(\abs{G}^2)\ \ {\rm as}\ \ \vert G\vert\to 0 \mbox{ uniformly with respect to } y\in Y.
  \end{equation*}

We next rescale the transverse variable $x_3$ in order to work on a fixed domain. Namely, we multiply the transverse variable by $h^{-1}$ so that the new variable belongs to the interval $I:=(-\frac{1}{2}, \frac{1}{2})$. We reassign the notation $x_3$ to the new variable (this should not cause any misunderstanding since we will always deal only with the rescaled domain in the future), so $x\in \Omega:= \omega\times I$. We also use the rescaled gradient
\begin{equation*}
 \nabla_h:= \left(\nabla',\, \frac{1}{h}\frac{\partial}{\partial x_3}\right):=\left(\frac{\partial}{\partial x_1},\,\frac{\partial}{\partial x_3},\, \frac{1}{h}\frac{\partial}{\partial x_3}\right),
\end{equation*}
so that the rescaled functional is given by
\begin{equation*}
E_h(u):= \int\limits_{\Omega} W(\e^{-1}x', \nabla_h u) dx,
\end{equation*}
where $u\in H^1(\Omega)$. We consider sequences of deformations $u_h$ in the bending regime, that is such that 
\begin{equation}\label{2.5}
 E_h(u_h) \leq C h^2.
\end{equation}
(In the non-rescaled setting this corresponds to the energies of order $h^3$.) In this paper we focus on the case when $h \ll  \e$. We consider $\e$ as a function of $h$, {\it {\it i.e.}} $h$ is our ``main'' parameter and $\e$ is the ``dependent'' parameter. 

The non-degeneracy property (\ref{2.2}) and the assumption (\ref{2.5}) imply that
\begin{equation}\label{2.7}
 \int\limits_{\Omega} \dist^2\bigl(\nabla_h u_h, \SO{3}\bigr)dx \leq C h^2.
\end{equation}
In what follows, we say that a sequence satisfying (\ref{2.7}) has finite bending energy.

\section{Two-scale compactness and second gradients}
\label{compactness}

The two cases when $\e$ is of order $h$ and $h \gg  \e$ were considered in \cite{HNV}. In these regimes the behaviour of the mid-surface of the plate is essentially macroscopic and the microscale only plays a role away from the mid-surface. In our case the microscopic behaviour of the mid-surface of the plate is very important. This is due to the fact that for each $\e$-cell the corresponding piece of the plate itself behaves like a miniature plate. This ``micro-plate'' behaviour becomes even more dominant when $h\ll \e^2$. In order to study this property we employ the method of two-scale convergence. The structure of the limit functional in the case of a homogeneous plate, namely, the fact that it is basically defined on the second gradient of the limit surface, suggests that the information about the microscopic behaviour should also appear in the form of a second gradient. The problem is that the original functional is defined only on $H^1$-functions, and how one would obtain the second derivatives with respect to the fast variable $y$ in the limit is not clear at the outset. However, while this is not possible in general, in the bending regime deformations possess an additional property which in a certain sense is equivalent to the existence of the second gradient, which we explain in detail in what follows.

Let us recall some results from \cite{FJM2002}. We denote 
\[
Y_{a,h}:= a + h Y, \,\, a \in h \mathbb Z^2,
\ \ \ \ \omega_h:= \bigcup_{Y_{a,h}\subset \omega}Y_{a,h}.
\]

\begin{theorem}\label{th3.1}
 Suppose that a sequence $u_h$ from $H^1(\Omega)$ has finite bending energy. Then: 
 \begin{enumerate}
  \item Up to a subsequence 
 \begin{equation*}
  \nabla_h u_h \to (\nabla' u,\,n) \mbox{ strongly in } L^2(\Omega),
 \end{equation*}
where $u=u(x')$ belongs to the space $H^2_{\rm iso}(\omega)$ of isometric immersions, {\it {\it i.e.}} such maps $u:\,\omega \to \mathbb R^3$ that $(\nabla' u)^\top \nabla' u = {\mathcal I}_2$, and $n:= u_{,1} \wedge u_{,2}$;
\item There exists a piecewise constant map $R_h : \, \omega_h \to \SO{3}$ such that
\begin{equation*}
 \int\limits_{\omega_h\times I} \left| \nabla_h u_h - R_h\right|^2 dx \leq C h^2;
\end{equation*}
\item The difference between the values of $R_h$ in each pair of neighbouring $h$-cells is small in the sense that
\begin{equation*}
  \int\limits_{\omega_h} \max_{\zeta \in \Upsilon_{x'}}\big| R_h(x'+\zeta) - R_h(x')\big|^2 dx' \leq C h^2,
\end{equation*}
where $\Upsilon_{x'}$ consists of those $\zeta\in\{-h,0,h\}^2$ for which $Y_{a+\zeta,h} \subset \omega_h$ for the value $a \in h \mathbb Z^2$ such that $x' \in Y_{a,h}$.
 \end{enumerate}
\end{theorem}

Apart from the strong compactness of the gradients $\nabla_h u_h$, which is true in our setting as well as for homogeneous plates, this theorem implies that oscillations of the gradients on the $h$-scale are bounded. The third property in Theorem \ref{th3.1} implies that the difference between values of $R_h$ in neighbouring $h$-cells divided by $h$ (which basically is the difference quotient) is bounded on average and a similar statement can be formulated for $\nabla_h u_h$ due to the second part of the theorem. Hence one can try to mollify $u_h,$ expecting that the second gradient of the mollification will be bounded. In what follows we always assume that $u_h$ is a sequence with finite bending energy.

Let $\varphi \in C^\infty_0(B_1),$ $\varphi\geq 0,$ be radially symmetric, where $B_1$ is a unit disc 
centred at the origin, $\int_{B_1} \varphi = 1$, and for each $h>0$ denote by $\varphi_h(x'):= h^{-2}\varphi(h^{-1}x'),$ $x'\in{\mathbb R}^2,$ the corresponding mollifier. Let $\omega'$ be a domain whose closure is contained in $\omega$. For each sufficiently small $h>0,$ consider the function $\tilde{u}_h=\widetilde{u}_h(x')$, $x'\in \omega'$, defined as the result of the simultaneous mollification of $u_h$ with $\varphi_h$ and averaging with respect to the variable $x_3$:
\begin{equation}\label{2.28}
 \widetilde{u}_h(x') : = \int\limits_I\int\limits_{\mathbb R^2} \varphi_h(\xi-x')u_h(\xi,x_3)\,d\xi dx_3 = \int\limits_I\int\limits_{\mathbb R^2} \varphi_h(\xi)u_h(\xi+x',x_3)\,d\xi dx_3, \ \ \ x'\in\omega.
\end{equation}
(The above integral is well defined in $\omega'$ for sufficiently small $h$ since in this case the support of $\varphi_h(\xi-x')$ is contained in $\omega.$) It is well known that the mollification $\widetilde{u}_h$ is infinitely smooth. We will need to estimate in $H^1$ the difference between $\widetilde{u}_h$ and the average of $u_h$ with respect to $x_3$,
\begin{equation*}
 \overline{u}_h(x'):=\int\limits_I u_h(x)dx_3,\ \ \ x'\in\omega',
\end{equation*}
as well as its derivatives up to the second-order. Such estimates are obtained in the next statement.

\begin{theorem}\label{l3.2}Let $\omega'$ be a domain whose closure is contained in $\omega$. The mollification (\ref{2.28}) satisfies the following inequalities with a constant $C>0$ independent of $\omega'$ and $h$:
\begin{equation}\label{3.17}
\|\widetilde{u}_h - \overline{u}_h\|_{H^1(\omega')}\leq C h,
\end{equation}
\begin{equation*}
  \|\nabla'^{\,2}\widetilde{u}_h\|_{L^2(\omega')} \leq C,
\end{equation*}
\begin{equation}\label{2.40}
\|\nabla'\widetilde u_h\|_{L^\infty(\omega')} \leq C.
\end{equation}
\end{theorem}
\begin{proof}
In the subsequent estimates we use a generalised Minkowski inequality (also sometimes referred to as the Minkowski integral inequality):
\begin{equation*}
\Bigg[\,\int\limits_{\mathcal Y} \Bigg|\int\limits_{\mathcal X} F(x,y) d\mu_{\mathcal X}(x)\Bigg|^p d\mu_{\mathcal Y}(y)\Bigg]^{1/p} \leq \int\limits_{\mathcal X}\Bigg[\,\int\limits_{\mathcal Y} \Bigg| F(x,y) \Bigg|^p d\mu_{\mathcal Y}(y)\Bigg]^{1/p}d\mu_{\mathcal X}(x),
\end{equation*}
where $\mathcal X$ and $\mathcal Y$ are measure spaces with measures $\mu_{\mathcal X}(x)$ and $\mu_{\mathcal Y}(y)$ respectively, $F:\,\mathcal X\times \mathcal Y \to \mathbb R$ is measurable and $1<p<\infty$, see {\it e.g.} \cite{HLP}. We obtain the $L^2$-estimate in (\ref{3.17}) as follows:
\begin{equation*}
\begin{gathered}
  \|\widetilde{u}_h - \overline{u}_h\|_{L^2(\omega')}= \Big(\int\limits_{\omega'}\Big(\int\limits_{\mathbb R^2} \varphi_h(\xi)\bigl(\overline{u}_h(\xi+x')-\overline{u}_h(x')\bigr)\,d\xi\Big)^2 dx'\Big)^{1/2}
  \\
  \leq \int\limits_{\mathbb R^2} \Big(\int\limits_{\omega'}\Big(\varphi_h(\xi)\bigl(\overline{u}_h(\xi+x')-\overline{u}_h(x')\bigr)\Big)^2 dx'\Big)^{1/2}\,d\xi \leq \max_{|\xi|\leq h}\big\|\overline{u}_h(\xi+x')-\overline{u}_h(x')\big\|_{L^2(\omega')}
  \\
  =\max_{|\xi|\leq h}\Big\|\int_0^1 \nabla' \overline{u}_h(t\xi+x')\cdot \xi \,dt\Big\|_{L^2(\omega')} \leq \max_{|\xi|\leq h}|\xi|\int_0^1\bigl\| \nabla' \overline{u}_h(t\xi+x')\bigr\|_{L^2(\omega')}\,dt 
  \\ 
  \leq C h\bigl\| \nabla' \overline{u}_h(x')\bigr\|_{L^2(\omega)}\leq C h.
\end{gathered}
\end{equation*}
Now we estimate the difference of the gradients of $\tilde{u}_h$ and $\overline{u}_h$. First notice that by the properties of mollification, $\nabla' \widetilde{u}_h$ equals to the mollification of $\nabla' \overline{u}_h$. Second, one has
\begin{equation}
\label{addeq}
 \|\nabla' \overline{u}_h - \widehat{R}_h\|_{L^2(\omega')} \leq \|\nabla' {u}_h - \widehat{R}_h\|_{L^2(\Omega')} \leq Ch,
\end{equation}
where $\widehat{R}_h$ stands for the first two columns of the matrix $R_h$ and $\Omega' := \omega' \times I$. The bound (\ref{addeq}) is a simple consequence of application of the generalised Minkowski inequality and the H\"{o}lder inequality. Further, the following estimate holds:
\begin{equation}
\label{2.79}
\begin{gathered}
 \|\nabla'\widetilde{u}_h - \nabla'\overline{u}_h\|_{L^2(\omega')} = \Big\|\int\limits_{\mathbb R^2} \varphi_h(\xi)\bigl(\nabla'\overline{u}_h(\xi+x')- \nabla'\overline{u}_h(x')\bigr)\,d\xi\Big\|_{L^2(\omega')}
 \\
 \leq \int\limits_{\mathbb R^2}\Big\| \varphi_h(\xi)\bigl(\nabla'\overline{u}_h(\xi+x')- \nabla'\overline{u}_h(x')\bigr)\Big\|_{L^2(\omega')}\,d\xi \leq \max_{|\xi|\leq h}\bigl\|\nabla'\overline{u}_h(\xi+x')-\nabla'\overline{u}_h(x')\bigr\|_{L^2(\omega')}
  \\
 \leq\max_{|\xi|\leq h}\bigl\|\nabla'\overline{u}_h(\xi+x')-\widehat{R}_h(x'+\xi)\bigr\|_{L^2(\omega')} + \max_{|\xi|\leq h}\bigl\|\widehat{R}_h(\xi+x')-\widehat{R}_h(x')
 \bigr\|_{L^2(\omega')} 
 \\
+\bigl\|\widehat{R}_h(x')-\nabla'\overline{u}_h(x')\bigr\|_{L^2(\omega')}\leq C h,
\end{gathered}
\end{equation}
where we use the generalised Minkowski inequality and Theorem \ref{th3.1}. 

Next we estimate the second gradient of $\widetilde{u}_h$, given by
\begin{equation*}
 \nabla'^{\,2}\widetilde{u}_h = \int\limits_{\mathbb R^2} \nabla'(\varphi_h(\xi-x'))\otimes\nabla'\overline{u}_h(\xi)\,d\xi = - \int\limits_{\mathbb R^2} \nabla'\varphi_h(\xi)\otimes\nabla'\overline{u}_h(x'+\xi)\,d\xi,
\end{equation*}
where $\nabla'\varphi_h(\xi) = h^{-3}\nabla'\varphi(h^{-1}\xi)$. Since the mollifier $\varphi$ is radially symmetric, $\partial_{\xi_1}\varphi_h(\xi)$ is an odd function of $\xi_1$ and an even function of $\xi_2$, and analogously for $\partial_{\xi_2}\varphi_h(\xi)$, hence we write
\begin{equation*}
 \partial_i\nabla'\widetilde{u}_h = - \int\limits_{\{|\xi|\leq h, \xi_i>0\}} \partial_i\varphi_h(\xi)\bigl(\nabla'\overline{u}_h(x'+\xi) - \nabla'\overline{u}_h(x'-\xi)\bigr)\,d\xi, \quad i=1,2.
\end{equation*}
Using the generalised Minkowski inequality once again we have, for $i=1,2$
\begin{equation}\label{2.33}
\begin{gathered}
  \|\partial_i\nabla'\widetilde{u}_h\|_{L^2(\omega')} \leq \int\limits_{\{|\xi|\leq h, \xi_i>0\}}\Bigl\|\partial_i\varphi_h(\xi)\bigl(\nabla'\overline{u}_h(x'+\xi) - \nabla'\overline{u}_h(x'-\xi)\bigr)\Bigr\|_{L^2(\omega')}\,d\xi
 \\
 \leq Ch^{-1} \max_{|\xi|\leq h}\bigl\|\nabla'\overline{u}_h(x'+\xi) - \nabla'\overline{u}_h(x'-\xi)\bigr\|_{L^2(\omega')} \leq C.
\end{gathered}
\end{equation}
The last inequality is obtained analogously to (\ref{2.79}).

Finally, we argue that $\nabla'\widetilde{u}_h$ is bounded pointwise independently of $h$:
\begin{equation*}
\begin{gathered}
\big|\nabla'\widetilde{u}_h(x') \big| = \Big|\int\limits_{\mathbb R^2} \varphi_h(\xi)\nabla'\overline{u}_h(\xi+x')\,d\xi \Big|\leq\bigl\|\varphi_h(\xi)\bigr\|_{L^2(B_h)} \|\nabla'\overline{u}_h(\xi+x')\|_{L^2(B_h)} 
\\
\leq C h^{-1}\bigl\|\nabla'\overline{u}_h(\xi+x')-R_h(\xi +x')\bigr\|_{L^2(B_h)} + C h^{-1}\bigl\|R_h(\xi +x')\bigr\|_{L^2(B_h)}\leq C,
\end{gathered}
\end{equation*}
for all points $x'\in\omega$ and sufficiently small values of $h$.
\end{proof}

As one can see from the above, $\nabla'\widetilde{u}_h$ is a ``reasonably good'' approximation of $\nabla' u_h$. It remains to approximate $h^{-1} \partial_3 u_h$ by a smooth vector function. Let $\widetilde{n}_h$ be the unit vector orthogonal to the surface $\widetilde{u}_h$:
\begin{equation*}
 \widetilde{n}_h := \partial_1 \widetilde{u}_h \wedge \partial_2 \widetilde{u}_h.
\end{equation*}
We show that the $L^2$-norm of $\widetilde{n}_h - R_{h,3}$ is at most of order $h,$ where $R_{h,i}$ stands for the $i$-th column of $R_h$. To this end we first transform the above expression for $\widetilde{n}_h$, as follows:
\begin{equation*}
\begin{gathered}
 \widetilde{n}_h = \partial_1 \widetilde{u}_h \wedge \partial_2 \widetilde{u}_h = \big(R_{h,1}+ (\partial_1 \widetilde{u}_h-R_{h,1})\big) \wedge \big(R_{h,2} + (\partial_2 \widetilde{u}_h - R_{h,2})\big) 
 \\
 = R_{h,3} + (\partial_1 \widetilde{u}_h-R_{h,1}) \wedge R_{h,2} + R_{h,1}\wedge (\partial_2 \widetilde{u}_h - R_{h,2}) + (\partial_1 \widetilde{u}_h-R_{h,1}) \wedge  (\partial_2 \widetilde{u}_h - R_{h,2}).
\end{gathered}
\end{equation*}
Since $R_h$ and $\nabla' \widetilde u_h$ are $L^\infty$-bounded, the $L^2$-norms of the last three terms are of order $h$, hence 
\begin{equation}\label{2.39}
\begin{gathered}
\|\widetilde{n}_h - R_{h,3}\|_{L^2(\omega')} \leq C h.
\end{gathered}
\end{equation}
Additionally, it is clear from (\ref{2.33}) and (\ref{2.40}) that 
\begin{equation}\label{2.41}
\begin{gathered}
\|\nabla' \widetilde{n}_h\|_{L^2(\omega')} \leq C.
\end{gathered}
\end{equation}

Putting together the above estimates for $\widetilde{n}_h$, Theorem \ref{l3.2} and Theorem \ref{th3.1} we obtain the following result.
\begin{proposition}\label{p3.3}
 Let $\omega'$ be a domain whose closure is contained in $\omega$. Under the assumptions of Theorem \ref{th3.1} the following estimates hold:
 \begin{equation*}
    \big\|(\nabla'\widetilde{u}_h\vert\,\widetilde{n}_h) - \nabla_h u_h\big\|_{L^2(\Omega')} \leq C h,
 \end{equation*}
 \begin{equation*}
    \big\|(\nabla'\widetilde{u}_h\vert\, \widetilde{n}_h) - R_h\big\|_{L^2(\omega')} \leq C h.
 \end{equation*}
\end{proposition}

As is customary in periodic homogenisation, and due to the fact that our problem is of two-scale nature, we require a suitable notion of convergence which takes into account the microscopic behaviour, {\it e.g.} the so-called two-scale convergence, introduced by Nguetseng in \cite{Nguetseng}.  Out of several equivalent definitions of two-scale convergence, we will use the one based on periodic unfolding, see {\it e.g.} \cite{CDG}. For convenience of the reader we give the definition of two-scale convergence that we use and some of its properties in Appendix \ref{twoscaleappendix}. Next we prove a general statement about the two-scale convergence of a second gradient, without any reference to a specific problem. This assertion may be used in various contexts where two-scale convergence of a second gradient is required. In the present paper we make use of this statement in Theorems \ref{th4.1} and \ref{th5.6}. In the remainder of this section we consider an arbitrary $d$-dimensional space, returning to the case $d=2$ in Section \ref{moderateregime} onwards.

\begin{theorem}[Second gradient and two-scale convergence]\label{th2.1}
Let $\omega\subset \mathbb R^d$ be a bounded domain and $Y:=[0,1)^d$. Assume that a sequence $v_\e\in H^2(\omega)$ is bounded in $H^2(\omega)$. Then up to a subsequence
 \begin{equation}\label{2.55}
 \nabla^2 v_\e \wtto \nabla^2 v(x) + \nabla_y^2 \varphi(x,y),
 \end{equation}
 where $v\in H^2(\omega)$ is the weak limit in $H^2(\omega)$ (hence, strong limit in $H^1(\omega)$) of the sequence $v_\e$ and 
 $\varphi \in L^2\bigl(\omega; H^2_{\rm per}(Y)\bigr)$ with zero mean with respect to $y\in Y$. 
 
Note that here $v_\e$ is a scalar function. The theorem directly applies to the vector case component-wise.
\end{theorem}
\begin{proof}
Following the idea of the proof of the two-scale convergence for the first gradients in \cite{CDG} we introduce the following function of $x$ and $y$:
 \begin{equation}\label{2.56}
  \widetilde{v}_\e(x,y):= \e^{-1}\Big(\e^{-1}\big(\mathcal T_\e(v_\e)-\mathcal M_Y(\mathcal T_\e(v_\e))\big) - y^c \cdot \mathcal M_Y(\mathcal T_\e(\nabla v_\e))\Big),
 \end{equation}
 where $\mathcal T_\e$ is the unfolding operator (see Appendix \ref{twoscaleappendix}), $M_Y(f) := \int_Y f dy$ is the operator of averaging with respect to the variable $y$, and $y^c:= y - (1/2,\ldots,1/2)^\top$. We calculate the first and second gradients of $\widetilde{v}_\e$ with respect to $y$:
 \begin{equation*}
  \nabla_y\widetilde{v}_\e= \e^{-1}\big(\mathcal T_\e(\nabla v_\e)- \mathcal M_Y(\mathcal T_\e(\nabla v_\e))\big),\ \ \ \ \ \ \ \ \ \ \ \ 
  \nabla_y^2\widetilde{v}_\e= \mathcal T_\e(\nabla^2 v_\e).
 \end{equation*}
 Notice that the  function $\widetilde{v}_\e$ and its gradient $\nabla_y\widetilde{v}_\e$ have zero mean value in $y$. Hence, we can apply the Poincar\'e inequality:
 \begin{equation*}
  \|\widetilde{v}_\e\|_{L^2(\omega\times Y)}\leq C\|\nabla_y\widetilde{v}_\e\|_{L^2(\omega\times Y)} \leq C\| \nabla_y^2\widetilde{v}_\e\|_{L^2(\omega\times Y)}\leq C,
 \end{equation*}
 where the last bound follows from the assumptions of the theorem. Hence, the sequence $\widetilde{v}_\e$ is bounded in $L^2\bigl(\omega; H^2(Y)\bigr)$ and it converges weakly in this space, up to a subsequence:
  \begin{equation*}
  \widetilde{v}_\e \rightharpoonup \widetilde{\varphi} \mbox{ weakly in } L^2\bigl(\omega; H^2(Y)\bigr).
 \end{equation*}
 Let a function $\varphi \in L^2\bigl(\omega; H^2(Y)\bigr)$ be such that 
  \begin{equation}
\widetilde{\varphi} = \frac{1}{2}\,y^c\cdot (\nabla^2 v)y^c + \varphi.
\label{triangle}
 \end{equation}
 Clearly, under this notation one has
   \begin{equation*}
  \nabla_y^2\widetilde{v}_\e \rightharpoonup \nabla^2 v(x) + \nabla_y^2 \varphi(x,y) \mbox{ weakly in } L^2(\omega\times Y),
 \end{equation*}
 which is equivalent to (\ref{2.55}). Now it only remains to prove the periodicity of $\varphi$ with respect to $y$. To this end, let $\psi(x,y')$ be an arbitrary function from $C_0^\infty(\omega\times Y')$, where $y':=(y_1,\ldots,y_{d-1})^\top$ and $Y':=[0,1)^{d-1}$. We write the difference between values of $\widetilde{v}_\e$ on the opposite faces of the cube $Y$ corresponding to $y_d=1$ and $y_d=0$ (the proof for other components of $y$ being exactly the same),
 \begin{equation*}
 \begin{gathered}
  \widetilde{v}_\e(x,y',1) - \widetilde{v}_\e(x,y',0) = \e^{-2}\big(\mathcal T_\e(v_\e)(x,y',1) - \mathcal T_\e(v_\e)(x,y',0)\big) - \e^{-1}\mathcal M_Y(\mathcal T_\e(\partial_d v_\e)) =
\\
= \e^{-2}\Big(\mathcal T_\e(v_\e)(x,y',1) - \mathcal T_\e(v_\e)(x,y',0) - \mathcal M_Y(\partial_{y_d}\mathcal T_\e(v_\e))\Big)=
 \\ 
 = \e^{-2}\Big(\mathcal T_\e(v_\e)(x,y',1) - \mathcal T_\e(v_\e)(x,y',0)- \int\limits_{Y'}\big(\mathcal T_\e(v_\e)(x,y',1) - \mathcal T_\e(v_\e)(x,y',0)\big)dy' \Big).
 \end{gathered}
 \end{equation*}
Notice that $\mathcal T_\e(v_\e)(x,y',1) = \mathcal T_\e(v_\e)(x+\e e_d,y',0)$, where $e_d$ is a unit vector in the direction of $x_d$ axis, hence we derive that 
 \begin{equation}\label{2.64}
 \begin{gathered}
  \int\limits_{\omega\times Y'}\big(\widetilde{v}_\e(x,y',1) - \widetilde{v}_\e(x,y',0)\big) \psi \, dx dy' = 
 \\ 
 = \int\limits_{\omega\times Y'} \e^{-2}\Big(\mathcal T_\e(v_\e)(x,y',0) - \int\limits_{Y'}\mathcal T_\e(v_\e)(x,y',0)dy' \Big) 
 \big(\psi(x-\e e_d,y') - \psi(x,y')\big) \, dx dy'
 \end{gathered}
 \end{equation}
The sequence $\e^{-1}(\psi(x-\e e_d,y') - \psi(x,y'))$ converges to $-\partial_d \psi$ strongly in $L^2(\omega\times Y')$. Let us denote 
\begin{equation*}
Z_\e(x,y):= \e^{-1}\big(\mathcal T_\e(v_\e)(x,y) - \int_{Y'}\mathcal T_\e(v_\e)(x,y)dy' \big).
\end{equation*}
Using a slight variation of the argument in Proposition 3.4 in \cite{CDG} one can see that $Z_\e$ converges strongly in $L^2(\omega; H^1(Y))$ to $(y^c)' \cdot \nabla' v$ as $\e\to 0$. Hence, its trace 
\begin{equation*}
 Z_\e(x,y',0)= \e^{-1}\big(\mathcal T_\e(v_\e)(x,y',0) - \int_{Y'}\mathcal T_\e(v_\e)(x,y',0)dy' \big)
\end{equation*}
converges to $(y^c)' \cdot \nabla' v$ weakly in $L^2(\omega\times Y')$, see Proposition \ref{p2.2} below. Passing to the limit in (\ref{2.64}) yields 
 \begin{equation}\label{3.40}
 \begin{gathered}
 \lim_{\e\to 0} \int\limits_{\omega\times Y'}\big(\widetilde{v}_\e(x,y',1) - \widetilde{v}_\e(x,y',0)\big) \psi \, dx dy' =  - \int\limits_{\omega\times Y'}(y^c)' \cdot \nabla' v \,\partial_d \psi \, dx dy'.
 \end{gathered}
 \end{equation}
We carry out the the same calculations for $z(x,y):=\frac{1}{2}\,y^c\cdot(\nabla^2 v)\, y^c,$ which results in
\begin{equation*}
  z(x,y',1) - z(x,y',0) = (y^c)' \cdot \partial_d\nabla' v.
\end{equation*}
Multiplying the last expression by $\psi$ and integrating by parts we get
 \begin{equation*}
 \begin{gathered}
\int\limits_{\omega\times Y'}\big(z(x,y',1) - z(x,y',0)\big) \psi \, dx dy' = - \int\limits_{\omega\times Y'}(y^c)'\cdot \nabla' v \,\partial_d \psi \, dx dy'.
 \end{gathered}
 \end{equation*}
Comparing the latter with (\ref{3.40}) we conclude that $\varphi(x,y)$ is periodic with respect to $y$. 

Analogously we prove the periodicity of $\nabla_y \varphi$. Indeed, we have 
 \begin{equation*}
 \begin{gathered}
\nabla_y\widetilde{v}_\e(x,y',1) - \nabla_y\widetilde{v}_\e(x,y',0) =  \e^{-1}\big(\mathcal T_\e(\nabla v_\e)(x,y',1) - \mathcal T_\e(\nabla v_\e)(x,y',0)\big).
 \end{gathered}
 \end{equation*}
Note that by the assumptions of the theorem and properties of two-scale convergence (see Appendix \ref{twoscaleappendix}) the sequence $\mathcal T_\e(\nabla v_\e)$ converges to $\nabla v$ strongly in $L^2(\omega; H^1(Y))$. Hence, by Proposition \ref{p2.2} below the trace $\mathcal T_\e(\nabla v_\e)(x,y',0)$ converges to $\nabla v$ weakly in $L^2(\omega\times Y')$, and therefore
 \begin{equation*}
 \begin{gathered}
 \lim_{\e\to 0} \int\limits_{\omega\times Y'}\big(\nabla_y\widetilde{v}_\e(x,y',1) - \nabla_y\widetilde{v}_\e(x,y',0)\big) \psi \, dx dy' 
\\
= \lim_{\e\to 0} \int\limits_{\omega\times Y'}\mathcal T_\e(\nabla v_\e)(x,y',0)\,\e^{-1}\big(\psi(x-\e e_d,y') - \psi(x,y')\big) \, dx dy' = 
 - \int\limits_{\omega\times Y'}\nabla v \,\partial_d \psi \, dx dy'.
 \end{gathered}
 \end{equation*}
Since $\nabla_y z(x,y',1) - \nabla_y z(x,y',0) = \partial_d \nabla v$ we have
 \begin{equation*}
 \begin{gathered}
\int\limits_{\omega\times Y'}\big(\nabla_y z(x,y',1) - \nabla_y z(x,y',0)\big) \psi \, dx dy' = - \int\limits_{\omega\times Y'}\nabla v \,\partial_d \psi \, dx dy',
 \end{gathered}
 \end{equation*}
which implies the periodicity of $\nabla_y \varphi$. Summarising all of the above we conclude that $\varphi \in L^2(\omega; H^2_{\rm per}(Y))$.
\end{proof}

\begin{proposition}\label{p2.2}
 Suppose that a sequence $f_n(x,y)$ converges to $f(x,y)$ weakly in $L^2(\omega; H^1(Y))$ and strongly in $L^2(\omega\times Y)$. Then its trace on $\omega\times \partial Y$ converges to the trace of $f$ weakly in $L^2(\omega\times \partial Y)$.
\end{proposition}
\begin{proof}
 Without loss of generality we will assume that $f=0$ and prove the assertion for the part of the boundary corresponding to $y_n=0$. We need to show that 
  \begin{equation}\label{2.70}
 \begin{gathered}
\lim_{n\to \infty}\int\limits_{\omega\times Y'}f_n(x,y',0) \psi(x,y') \, dx dy' = 0,
 \end{gathered}
 \end{equation}
 for all $\psi\in C^\infty_0(\omega\times Y')$. Let us fix $\psi$ and consider it as a function defined on $\omega\times Y$ via $\psi(x,y):=\psi(x,y')$. Let us define function $h_n(y)$ as $\int_\omega f_n \psi\, dx$. It is easy to see that $h_n \in H^1(Y)$ and $\|h_n\|_{L^2(Y)}\leq C\|f_n\|_{L^2(\omega\times Y)}$, $\|h_n\|_{H^1(Y)}\leq C\|f_n\|_{L^2(\omega; H^1(Y))}$. Hence $h_n\to 0$ weakly in $H^1(Y)$ and strongly in $L^2(Y)$. By the trace theorem $h_n(y',0) \to 0$ strongly in $L^2(Y')$. In particular, $\int_{Y'} h_n(y',0) \, dy' \to 0$. But $h_n(y',0) = \int_\omega f_n(x,y',0) \psi(x,y')\, dx$, hence (\ref{2.70}) follows.
\end{proof}

Now, we have all necessary tools to proceed to the $\Gamma$-convergence statement.

\section{Limit functional and $\Gamma$-convergence for the ``moderate'' thickness-to-period ratio $\e\gg h\gg \e^2$}
\label{moderateregime}

Our main result is establishing $\Gamma$-convergence of the rescaled functional $E_h$ to the limit homogenised functional defined on isometric surfaces. $\Gamma$-convergence is a standard tool to describe the asymptotic behaviour of nonlinear functionals, introduced by De Giorgi \cite{DeGiorgiFranzoni}. The general definition of $\Gamma$-convergence is the following:
\begin{definition}
Let $F_i$, $i=1,2,\ldots$ and $F_\infty$ be some functionals defined on functional spaces $\mathcal H_i$ and $\mathcal{H}_\infty$. Then we say that the sequence of functionals $F_i$ $\Gamma$-converges to the functional $F_\infty$,
\begin{equation*}
F_\infty = \Gamma-\lim_{i\to \infty} F_i,
\end{equation*} 
if the following two properties hold:
\begin{enumerate}

 \item {\bf Lower bound.} For any sequence $u_i \in \mathcal H_i$ converging in a certain sense to $u_\infty \in \mathcal H_\infty$ one has
\begin{equation*}
 \liminf_{i\to \infty} F_i(u_i) \geq F_\infty(u_\infty).
\end{equation*}

\item {\bf Recovery of the lower bound.}  For any $u_\infty \in \mathcal H_\infty$ with $F_\infty(u_\infty)< \infty$ there exists a sequence $u_i \in \mathcal H_i$ converging to $u_\infty$ in a certain sense such that 
\begin{equation*}
  \lim_{i\to \infty} F_i(u_i) = F_\infty(u_\infty).
\end{equation*}
\end{enumerate}
\end{definition}
In the above definition the notion of convergence depends on the underlying properties of the functionals and may be quite non-standard, in particular, because the spaces $\mathcal H_i$ and $\mathcal H_\infty$ often do not coincide, as in the present paper, for example. Moreover, the notion of convergence in the second part of the definition may be different from the convergence in the first part. In this case the former must normally be ``stronger'' that the latter in order to facilitate the main purpose of the definition (nevertheless, this requirement is not always critical, {\it e.g.} see \cite{CherCher}). The lower bound inequality ensures that the limit functional gives in the limit a lower bound for the values of the  functionals $F_i$. The recovery property implies that the lower bound $F_\infty$ is the greatest lower bound. In particular, the definition implies that 
\begin{equation*}
\lim_{i\to\infty} (\inf F_i) = \inf F_\infty,
\end{equation*}
and, if a sequence $u_i$ of ``almost minimisers'' of $F_i$ converges to $u_\infty$, then the latter is a minimiser of $F_\infty$.

In the present work we obtain different $\Gamma$-limits for the regimes $\e\gg h\gg \e^2$ and $\e^2\gg h$. In this section we present our results for the ``moderate'' regime $\e\gg h\gg \e^2$, for which the limit homogenised functional is given by 
\begin{equation*}
E_{\rm hom}^{\rm m}(u):=\frac{1}{12}\int\limits_{\omega} Q_{\rm hom}^{\rm m}({\rm II}) dx', \,\, u \in H^2_{\rm iso}(\omega),
\end{equation*}
where ${\rm II}$ is the matrix of the second fundamental form of $u$,
\begin{equation*}
 {\rm II} : = (\nabla' u)^\top \nabla' n, \quad n=\partial_1 u \wedge \partial_2 u,
\end{equation*}
and the homogenised stored-energy function is given by
\begin{equation}\label{2.49}
\begin{gathered}
Q_{\rm hom}^{\rm m}({\rm II}) : = \min_{\psi \in H^2_{\rm per}(Y)} \int\limits_Y Q_2\bigl(y,{\rm II} + \nabla_y^2 \psi(y)\bigr) dy,
 \end{gathered}
\end{equation}
where $Q_2$ is the quadratic form defined on $2\times 2$ matrices and is obtained from the form $Q_3$ by minimising when reducing to two dimensions (notice that $Q_3$ vanishes on skew-symmetric matrices):
\begin{equation}\label{58}
Q_2(y, A):= \min_{a\in \mathbb R^3} Q_3(y, {A} +e_3 \otimes a + a \otimes e_3) = \min_{a\in \mathbb R^3} Q_3(y, {A}  + a \otimes e_3).
\end{equation}
where we implicitly assume that the space $\mathbb R^{2\times 2}$ matrices is naturally embedded in the space $\mathbb R^{3\times 3}$ by adding zero third column and zero third row, {\it cf.} \cite{FJM2002}.

\begin{theorem}\label{th4.1}
Suppose that $\e\gg h\gg \e^2$, {\it i.e.} $\e^{-1} h = o(1)$ and $h^{-1}\e^2=o(1)$ as $h\to 0$. Then the rescaled sequence of functionals $h^{-2} E_h$ $\Gamma$-converges to the limit functional $E_{\rm hom}^{\rm m}$ in the following sense.
\begin{enumerate}
 \item {\bf Lower bound.} For every bounded bending energy sequence $u_h\in H^1(\Omega)$ such that $\nabla' u_h$ converges to $\nabla' u$ weakly in $L^2(\Omega)$, $u\in H^2_{\rm iso}(\omega)$, the lower semicontinuity type inequality holds:
 \begin{equation*}
  \liminf_{h\to 0} h^{-2} E_h(u_h) \geq E_{\rm hom}^{\rm m}(u).
 \end{equation*}
\item {\bf Recovery of the lower bound.} For every $u\in H^2_{\rm iso}(\omega)$ there exists a sequence $u_h^{\rm rec}\in H^1(\Omega)$ such that $\nabla' u_h^{\rm rec}$ converges to $\nabla' u$ strongly in $L^2(\Omega)$ and
 \begin{equation*}
  \lim_{h\to 0} h^{-2} E_h(u_h^{\rm rec}) = E_{\rm hom}^{\rm m}(u).
 \end{equation*}
\end{enumerate}
\end{theorem}

\begin{proof}{\bf Lower bound.}  Since $\nabla_h u$ is close to a rotation-valued function $R_h$ we can use the frame indifference property of the stored energy function $W$ and its Taylor expansion near identity to linearise the functional. Roughly speaking, we will use the observation that 
\begin{equation}\label{4.63}
 W(y, \nabla_h u_h) = W(y, R_h^\top \nabla_h u_h) = W\bigl(y, {\mathcal I}_3 + (R_h^\top \nabla_h u_h - {\mathcal I}_3)\bigr) \approx Q_3(y, h G_h),
\end{equation}
where
\begin{equation*}
 G_h:= h^{-1}\bigl( R_h^\top \nabla_h u_h(x) - {\mathcal I}_3\bigr).
\end{equation*}
However, the last asymptotic equality in (\ref{4.63}) has to be made precise, since the smallness of the term $h G_h$ is not pointwise. We also need to understand the structure of the (weak) two-scale limit of $G_h$. To this end we will use a special representation of $u_h$. Let $\omega'$ be a domain whose closure is contained in $\omega$ and $\Omega'=\omega'\times I$. For small enough $h$ we can write $u_h$ in $\Omega'$ in the form
\begin{equation*}
\begin{gathered}
u_h(x) = \widetilde{u}_h(x') + h x_3 \widetilde{n}_h(x') + h z_h(x),
\end{gathered}
\end{equation*}
where $\widetilde u_h$ is the mollified mid-surface of $u_h$ defined in the previous section and $\widetilde{n}_h$ is the corresponding normal. 
Then one has
\begin{equation}\label{2.43}
\begin{gathered}
G_h = h^{-1} R_h^\top\bigl((\nabla'\widetilde{u}_h|\widetilde{n}_h) - R_h\bigr) +  x_3 R_h^\top(\nabla' \widetilde{n}_h|0) +  R_h^\top\nabla_h z_h.
\end{gathered}
\end{equation}
Notice that $G_h$ is bounded in $L^2(\Omega')$ by Theorem \ref{th3.1}. The first term on the right-hand side is bounded due to Proposition \ref{p3.3}. The second term on the right-hand side is bounded by  (\ref{2.41}). Hence, the third term is also bounded.

Since we perform the dimension reduction when passing to the limit, we are interested only in the behaviour of the top left $2\times 2$ block of the matrix $G_h$. Let us use the following notation: if $A\in \mathbb R^{3\times 3}$, then $\widehat A$ denotes the first two columns of $A$, and $\overline A$ is obtained from $A$ by omitting the third column and the third row. We have
\begin{equation}\label{4.67}
\begin{gathered}
\overline G_h = h^{-1} \widehat R_h^\top(\nabla'\widetilde{u}_h - \widehat R_h) +  x_3 \widehat R_h^\top\nabla' \widetilde{n}_h +  \widehat R_h^\top\nabla' z_h.
\end{gathered}
\end{equation}

The limit of the first and the third terms on the right-hand side of (\ref{4.67}) is of no importance for us apart from the fact that they do not depend on $x_3$. This is obvious in case of the first term. As for the third term, let $M_{Y\times I}\bigl(\mathcal T_\e(z_h)\bigr):=\int_{Y\times I}\mathcal T_\e(z_h)dydx_3$ and consider 
\begin{equation*}
 w_h := \e^{-1}\bigl(\mathcal T_\e(z_h) - M_{Y\times I}(\mathcal T_\e(z_h))\bigr).
\end{equation*}
It is easy to see that $\nabla_y w_h = \mathcal T_\e(\nabla' z_h)$ and $\partial_{3} w_h = \partial_{3} \e^{-1} \mathcal T_\e(z_h) = \e^{-1}h \mathcal T_\e(h^{-1}\partial_{3} z_h)$. Hence we have $\|\nabla_{y,x_3}w_h\|_{L^p(\Omega'\times Y)}\leq C$. Since $M_{Y\times I}(w_h) = 0$ we can use the Poincar\'e inequality $\|w_h\|_{L^p(\Omega'\times Y)}\leq C \|\nabla_{y,x_3}w_h\|_{L^p(\Omega'\times Y)}$. So $w_h$ is bounded in $L^p(\omega';W^{1,p} (Y\times I))$. Then, up to a subsequence, we have
\begin{equation*}
 w_h \rightharpoonup w(x,y) \mbox{ weakly in } L^p\bigl(\omega';W^{1,p} (Y\times I)\bigr).
\end{equation*}
On the other hand we know that $\partial_{3} w_h$ vanishes in the limit, hence $w$ is independent of $x_3$:
\begin{equation}\label{1.4}
 w=w(x',y).
\end{equation}
Therefore, we have
\begin{equation}\label{4.71}
 h^{-1} \widehat R_h^\top(\nabla'\widetilde{u}_h - \widehat R_h)  +  \widehat R_h^\top\nabla' z_h \wtto H(x',y)
\end{equation}
for some $H\in L^2(\omega'\times Y)$, up to a subsequence.

Now we examine the expression $(\widehat{R}_h)^\top \nabla' \widetilde{n}_h$. The element $(1,1)$ of this matrix can be written as
\begin{equation*}
\begin{gathered}
R_{h,1} \cdot \partial_1 \widetilde{n}_h = R_{h,1} \cdot \partial_{11} \widetilde{u}_h \wedge \partial_{2} \widetilde{u}_h + R_{h,1} \cdot \partial_{1} \widetilde{u}_h \wedge \partial_{12} \widetilde{u}_h =
\\
= \partial_{11} \widetilde{u}_h \cdot \partial_{2} \widetilde{u}_h \wedge R_{h,1} + \partial_{12} \widetilde{u}_h \cdot R_{h,1} \wedge \partial_{1} \widetilde{u}_h =
\\
= - \partial_{11} \widetilde{u}_h \cdot R_{h,3} + \partial_{11} \widetilde{u}_h \cdot (\partial_{2} \widetilde{u}_h -R_{h,2}) \wedge R_{h,1} + \partial_{12} \widetilde{u}_h \cdot R_{h,1} \wedge (\partial_{1} \widetilde{u}_h -  R_{h,1}),
\end{gathered}
\end{equation*}
where the vector product operations take priority over the inner products. Consider the second term in the last expression (the third term can be dealt in the same way). By Theorem \ref{th3.1} and Theorem \ref{l3.2} one has 
\[
\|\partial_{11}\widetilde{u}_h\|_{L^2(\omega')}\leq C,\ \ \ \|\partial_{2}\widetilde{u}_h - R_{h,2}\|_{L^2(\omega')}\leq C h'\ \ \ \|\partial_{2}\widetilde{u}_h - R_{h,2}\|_{L^\infty(\omega')}\leq C. 
\]
On the one hand, the second term is bounded in $L^2(\omega')$, and so converges weakly two-scale to some function from $L^2(\omega'\times Y)$ up to a subsequence; on the other hand, for any  $C^\infty_0(\omega'\times Y)$ test-function we get zero in the limit. We conclude that the weak two-scale limits of the second and the third terms are zero. We can apply analogous argument to the other elements of $(\widehat{R}_h)^\top \nabla' \widetilde{n}_h$. Therefore we see that $(\widehat{R}_h)^\top \nabla' \widetilde{n}_h$ and $-\nabla'^{\,2} \widetilde{u}_h \cdot R_{h,3}$ have the same weak two-scale limit (the ``dot'' product gives a $2\times 2$ matrix whose elements are $-\partial_{ij}  \widetilde{u}_h \cdot R_{h,3}$).   
Applying Theorem \ref{th2.1} to $\widetilde{u}_h$ one has 
\begin{equation*}
\begin{gathered}
\nabla'^{\, 2} \widetilde{u}_h \twconv \nabla'^{\, 2} u(x') + \nabla_y^2 \varphi(x',y)
\end{gathered}
\end{equation*}
up to a subsequence, for some $\varphi\in L^2(\Omega'; H^2_{\rm per}(Y))$, whereas $R_{h,3}$ converges strongly to $n$ due to Theorem \ref{th3.1}. Hence, 
\begin{equation}\label{4.74}
\begin{gathered}
(\widehat{R}_h)^\top \nabla' \widetilde{n}_h \twconv \big(-\nabla'^{\, 2} u + \nabla_y^2 \varphi\big)\cdot n = {\rm II} + \nabla_y^2 \psi,
\end{gathered}
\end{equation}
where $\psi:= - \varphi \cdot n \in L^2\bigl(\Omega'; H^2_{\rm per}(Y)\bigr)$. Putting together (\ref{4.71}) and (\ref{4.74}) we conclude that up to a subsequence
\begin{equation}\label{4.75}
\begin{gathered}
\overline{G}_h \twconv x_3\bigl({\rm II}(x') + \nabla_y^2 \psi(x',y)\bigr) + H(x',y).
\end{gathered}
\end{equation}

We can proceed to the last step of the proof. In order to use the Taylor expansion of $W$ we need to control $G_h$ pointwise. Let $\chi_h$ be a characteristic function of the subset of $\Omega'$ on which $h G_h$ is relatively small, namely,
\begin{equation}\label{4.76}
\chi_h := \left\{ \begin{array}{ll} 
				1 & \mbox{if } |h G_h| \leq h^{1/2},
				\\
				0 & \mbox{otherwise}.
                             \end{array}\right.
\end{equation}
It is easy to see that $\int_{\Omega'}(1-\chi_h)dx \to 0$. Since $W$ is non-negative we can restrict the integral to $\Omega'$ and discard the set of points on which $|h G_h| > h^{1/2}$ by using $\chi_h$:
\begin{equation}\label{4.77}
\begin{gathered}
 h^{-2}E_h(u_h) = h^{-2}\int\limits_{\Omega} W(\e^{-1}x', \nabla_h u_h) dx \geq  h^{-2}\int\limits_{\Omega'} W(\e^{-1}x', \chi_h\nabla_h u_h) dx
\\
=  h^{-2}\int\limits_{\Omega'} W\big(\e^{-1}x', \chi_h({\mathcal I}_3 + h G_h) \big)dx =  \int\limits_{\Omega'} Q_3(\e^{-1}x', \chi_h G_h )dx +  h^{-2}\int\limits_{\Omega'} r(\e^{-1}x', \chi_h h G_h)dx.
 \end{gathered}
\end{equation}
By the assumptions on $W$ we have $\big|r(\e^{-1}x', \chi_h h G_h)\big| \leq g(\chi_h |h G_h|^2)$ for some function $g(t) = o(t)$. An elementary lemma below implies that 
\begin{equation}\label{4.78}
h^{-2}\int\limits_{\Omega'} r\bigl(\e^{-1}x', \chi_h h G_h\bigr)dx\stackrel{h\to0}\longrightarrow0.
\end{equation}

\begin{lemma}
\label{simplelemma}
Suppose that $t^{-1}g(t)\to0$ as $t\to0,$ and let $\{f_t\}\subset L^1(\Omega')$ be a family of non-negative functions such that $t^{-1}\int_{\Omega'} f_t$ is bounded with respect to $t$ and $\|f_t\|_{L^\infty(\Omega')}\to0$ as $t\to0.$ Then the convergence
 \begin{equation*}
t^{-1}\int\limits_{\Omega'}g\bigl(f_t(x)\bigr)dx\stackrel{t\to0}\longrightarrow 0
 \end{equation*}
holds.
\end{lemma}
\begin{proof}
Notice that
 \begin{equation*}
  \int\limits_{\Omega'} \frac{g(f_t(x))}{t}dx = \int\limits_{\Omega'} \frac{g(f_t(x))}{f_t(x)} \frac{f_t(x)}{t} dx \leq \left(\sup_{s\in(0,\,\|f_t\|_{L^\infty(\Omega')})}\frac{g(s)}{s}\right)\, \int\limits_{\Omega'} \frac{f_t(x)}{t} dx\to 0.
 \end{equation*}
Here it is implied, with some abuse of notation, that $(f_t(x))^{-1}g(f_t(x)) = 0$ whenever $f_t(x)=0$. 
\end{proof}
We apply Lemma \ref{simplelemma} to the expression (\ref{4.78}) by setting $t:=h^2$, $f_t:= \chi_h |h G_h|^2$. Notice that the particular choice of the cut-off threshold in (\ref{4.76}) is not important: instead of $h^{1/2}$ one can take any threshold that goes to 
zero slower than $h.$ 

Now we estimate the quadratic term in (\ref{4.77}). By definition of the form $Q_2$ ant the properties of the unfolding operator we have
\begin{equation}\label{4.81}
\begin{gathered}
\int\limits_{\Omega'} Q_3(\e^{-1}x', \chi_h G_h )dx \geq \int\limits_{\Omega'} Q_2(\e^{-1}x', \chi_h \overline G_h)dx \geq  \int\limits_{\Omega'}\int\limits_Y Q_2\big(y, \mathcal T_\e(\chi_h \overline G_h)\big) dy dx 
 \end{gathered}
\end{equation}
It is well known that convex integral functionals are lower semi-continuous with respect to the weak convergence in functional spaces, see e.g. \cite{Dacorogna}. Assume that we deal with a weakly converging (as in (\ref{4.75})) subsequence $\overline G_h$. Then passing to the limit and recalling that $\chi_h$ converges strongly to $1$ we obtain the following inequality:
\begin{equation}
\begin{gathered}
\liminf_{h\to 0}\int\limits_{\Omega'}\int\limits_Y Q_2\big(y, \mathcal T_\e(\chi_h \overline G_h)\big) dy dx \geq \int\limits_{\Omega'}\int\limits_Y Q_2\big(y, x_3 ({\rm II}(x') + \nabla_y^2 \psi(x',y)) + H(x',y)\big) dy dx
 \\
 =\frac{1}{12}\int\limits_{\omega'}\int\limits_Y Q_2(y,{\rm II} + \nabla_y^2 \psi) dy dx' + \int\limits_{\omega'}\int\limits_Y Q_2(y,H) dy dx' 
 \\
 \geq\frac{1}{12}\int\limits_{\omega'}\int\limits_Y Q_2(y,{\rm II} + \nabla_y^2 \psi) dy dx'\geq \frac{1}{12}\int\limits_{\omega'}\int\limits_Y Q_{\rm hom}^{\rm m}({\rm II}) dy dx'.
 \end{gathered}
 \label{commonest}
\end{equation}
Note that in the second step in the above we integrate with respect $x_3\in I$, which causes the cross term of $x_3 ({\rm II} + \nabla_y^2 \psi)$ and $H$ in the quadratic form $Q_2$ to disappear. At the last step we just used the definition of $Q_{\rm hom}^{\rm m}$. 

A simple argument by contradiction shows that the inequality
\begin{equation}\label{4.83}
\begin{gathered}
\liminf_{h\to 0}\int\limits_{\Omega'}\int\limits_Y Q_2\big(y, \mathcal T_\e(\chi_h \overline G_h)\big) dy dx \geq  \frac{1}{12}\int\limits_{\omega'}\int\limits_Y Q_{\rm hom}^{\rm m}({\rm II}) dy dx'
 \end{gathered}
\end{equation}
holds for the whole sequence $\overline G_h$. Putting together (\ref{4.77}), (\ref{4.78}), (\ref{4.81}) and (\ref{4.83}) and noting that $\omega'$ is an arbitrary subset of $\omega$ concludes the proof of the lower bound.

\begin{remark}
The above proof applies to the general case of $h\ll\varepsilon,$ up to the last inequality in (\ref{commonest}). We will use this observation in our analysis of the lower bound in the case $h\ll\varepsilon^2$ in Section \ref{extremeenergy}.
\end{remark}

\medskip

{\bf Recovery sequence.} In order to build a recovery sequence one normally needs to ensure that the argument of the stored-energy density $W$ is an $L^\infty$ function, so that one can pass to the limit using the Taylor expansion (\ref{WTaylor}). If this condition is not in place one would have to estimate $W$ on the ``bad'' set where the argument might be uncontrollably large, which would require additional restrictions on $W$. For example in (\cite{FJM2002}) the authors used a truncation result for Sobolev maps, which consists of changing a function on some subset to make it together with its gradient essentially bounded and estimating the measure of this set. In this paper we will use a result of Pakzad from \cite{Pakzad} (which appeared after \cite{FJM2002}). It states that the space of infinitely smooth isometric immersions from a convex set $\overline\omega \subset \mathbb R^2$ into $\mathbb R^3$ is strongly dense in the Sobolev space of isometric immersion $H^2_{\rm iso}(\omega)$. 
Let $u\in H^2_{\rm iso}(\omega)\cap C^\infty(\overline\omega)$ and ${\rm II}$ be the matrix of the corresponding second fundamental form. Let also $\psi(x',y)$ and $d(x',y)$ be correspondingly a scalar and a vector functions from $C^\infty (\overline{\omega} \times \overline{Y})$ periodic with respect to $y$. 

Denote $\psi_\e:= \psi(x',\e^{-1}x')$, $d_\e:=d(x',\e^{-1}x')$ and consider the sequence
\begin{equation}\label{4.84}
\begin{gathered}
u_h = u -\e^2 \psi_\e n + h x_3 n_\e + h^2 \frac{x_3^2}{2} R d_\e,
\end{gathered}
\end{equation}
where $n:=\partial_1 u\wedge \partial_2 u$, $n_\e:=n + \e \partial_{y_1} \psi_\e \partial_1 u + \e \partial_{y_2} \psi_\e \partial_2 u$ and $R:=(\nabla'u|n)$. Calculating the rescaled gradient of $u_h$ we obtain
\begin{equation*}
\begin{gathered}
\nabla_h u_h = (\nabla'u -\e \nabla_y(\psi_\e n) | n_\e)+ h x_3 (\nabla'n + \nabla_y (\partial_{y_1} \psi_\e \partial_1 u + \partial_{y_2} \psi_\e \partial_2 u)|R d_\e) + U_h,
\end{gathered}
\end{equation*}
where $U_h$ denotes the rest of the terms:
\begin{equation*}
\begin{gathered}
U_h := - \e^2 (\nabla'(\psi_\e n) | 0)+ \e h x_3 (\nabla' (\partial_{y_1} \psi_\e \partial_1 u + \partial_{y_2} \psi_\e \partial_2 u)|0) + h^2 \frac{x_3^2}{2} \big((\nabla ' + \e^{-1} \nabla_y) (R d_\e)|0\big).
\end{gathered}
\end{equation*}
Due to our assumptions we see that $U_h = O(\e^2) = o(h)$ uniformly on $\omega$. A simple calculation shows that 
\begin{equation}\label{4.87}
\begin{aligned}
(\nabla'u -\e \nabla_y(\psi_\e n) | n_\e)^\top (\nabla'u &-\e \nabla_y(\psi_\e n) | n_\e) 
\\
= &{\mathcal I}_3+ \e^2 \left(\begin{array}{ccc}
		 (\partial_{y_1} \psi_\e)^2 & \partial_{y_1} \psi_\e \partial_{y_2} \psi_\e\ \ & 0
		 \\
		 \partial_{y_1} \psi_\e\partial_{y_2} \psi_\e\ \  & (\partial_{y_2} \psi_\e)^2 & 0
		 \\
		 0 & 0 & |\nabla_y \psi_\e|^2
                \end{array}
\right)
= {\mathcal I}_3 + o(h).
\end{aligned}
\end{equation}
This implies that there exists an infinitely smooth function $R_h$ with values in $\SO 3$ such that
\begin{equation*}
\begin{aligned}
R_h^\top (\nabla'u -\e \nabla_y(\psi_\e n) | n_\e) = {\mathcal I}_3 + o(h).
\end{aligned}
\end{equation*}
(It can be constructed explicitly via Gram--Schmidt orthonormalisation for the columns of the matrix $(\nabla'u -\e \nabla_y(\psi_\e n) | n_\e).$)

Further, consider the expression
\begin{equation*}
\begin{aligned}
G_h := h^{-1} (R_h^\top \nabla_h u_h - {\mathcal I}_3) = x_3 R_h^\top (\nabla'n + \nabla_y (\partial_{y_1} \psi_\e \partial_1 u + \partial_{y_2} \psi_\e \partial_2 u)|R d_\e) + o(1).
\end{aligned}
\end{equation*}
Notice that $R_h$ converges to $R$  uniformly on $\omega$, $\nabla_y^2 \psi_\e$ converges strongly two-scale to $\nabla_y^2 \psi,$ and $d_\e$  converges strongly two-scale to $d$. Hence, the convergence
\begin{equation}\label{4.90}
\begin{aligned}
G_h \stto x_3 ({\rm II} + \nabla_y^2 \psi \,|\, d)
\end{aligned}
\end{equation}
holds. Passing to the limit in
\begin{equation*}
\begin{gathered}
  h^{-2} E_h(u_h) = h^{-2}\int\limits_{\Omega} W(\e^{-1}x', \nabla_h u_h) dx =  \int\limits_{\Omega} Q_3(\e^{-1}x', G_h )dx +  h^{-2}\int\limits_{\Omega} r(\e^{-1}x', h G_h)dx
 \end{gathered}
\end{equation*}
we see that the remainder term vanishes since $G_h$ is uniformly bounded in $L^\infty (\omega)$, and for the quadratic term due to (\ref{4.90}) we have
\begin{equation*}
\begin{aligned}
\int\limits_{\Omega} Q_3(\e^{-1}x', G_h )dx = \int\limits_{\Omega\times Y} Q_3\bigl(y, \mathcal T_\e (G_h)\bigr)dx dy + \int\limits_{\Lambda_\e\times I} Q_3(\e^{-1}x', G_h )dx 
\\
\to \frac{1}{12} \int\limits_{\omega\times Y} Q_3\big(y,({\rm II} + \nabla_y^2 \psi \,|\, d) \big)dx' dy
 \end{aligned}
\end{equation*}
(for the definition of $\Lambda_\e$ see Appendix \ref{twoscaleappendix}). Thus 
\begin{equation}\label{4.93}
\begin{gathered}
 \lim_{h\to 0}  h^{-2} E_h(u_h) =  \frac{1}{12} \int\limits_{\omega\times Y} Q_3\big(y,({\rm II} + \nabla_y^2 \psi \,|\, d) \big)dx' dy.
 \end{gathered}
\end{equation}

We conclude the proof with a standard argument showing that the above assumption of smoothness of $u,$ $\psi,$ $d$ is not restrictive. Let $u$ be an arbitrary element of $H^2_{\rm iso}(\omega)$ and take $\psi(x',y)\in L^2(\omega; H^2_{\rm per}(Y))$ to be a solution to the minimisation problem in (\ref{2.49}) corresponding to the matrix of the second fundamental form of $u$, that is
\begin{equation*}
\begin{gathered}
Q_{\rm hom}^{\rm m}({\rm II})  = \int\limits_Y Q_2(y,{\rm II} + \nabla_y^2 \psi) dy.
 \end{gathered}
\end{equation*}
Let also $d(x',y)$ be such that
\begin{equation*}
\begin{gathered}
 Q_2(y,{\rm II} + \nabla_y^2 \psi) = Q_3\big(y,({\rm II} + \nabla_y^2 \psi \,|\, d) \big).
\end{gathered}
\end{equation*}
There exist sequences of $C^\infty$ functions $u_j$, $\psi_j$ and $d_j$ converging to $u$, $\psi$ and $d$ in the corresponding function spaces such that 
\begin{equation*}
\begin{gathered}
\left|\,\int\limits_{\omega\times Y} Q_3\big(y,({\rm II}_j + \nabla_y^2 \psi_j \,|\, d_j) \big) dx' dy - \int\limits_{\omega} Q_{\rm hom}^{\rm m}({\rm II}) dx' \right|\leq \frac{1}{j}.
\end{gathered}
\end{equation*}
For each $j$ define $u_h^{(j)}$ via formula (\ref{4.84}) with $u$, $\psi$ and $d$ replaced by $u_j$, $\psi_j$ and $d_j$. For each sequence $u_h^{(j)}$ the analogue of convergence (\ref{4.93}) holds. It follows that there exists a subsequence $h_j \to 0$ such that
\begin{equation*}
\begin{gathered}
\left| E_h(u_{h_j}^{(j)}) - \int\limits_{\omega} Q_{\rm hom}^{\rm m}({\rm II}) dx' \right|\leq \frac{2}{j}.
\end{gathered}
\end{equation*}
Hence $u_h^{\rm rec}:=u_{h_j}^{(j)}$ is the required recovery sequence.
\end{proof}

\section{A microscale isometry constraint on the limits of finite bending energy sequences in the ``supercritical'' regime $h\ll \e^2$}\label{s5}

During the process of constructing a recovery sequence we have encountered a problem in the case when $h\ll \e^2$. Namely, introducing the corrector terms $-\e^2 \psi_\e n$ and $\e \partial_{y_1} \psi_\e \partial_1 u + \e \partial_{y_2} \psi_\e \partial_2 u$ in (\ref{4.84}) incurs an error of order $\e^2$ ({\it c.f.} (\ref{4.87})) which becomes of order $h^{-1}\e^2$ when scaled in the subsequent argument. This error does not vanish if $h\ll \e^2$ or $h\sim \e^2$. In order to understand the nature of the problem we write a formal asymptotic expansion for a recovery sequence, see Appendix \ref{asympsection}. In the process of eliminating the errors of order $\e^2$ and higher, it transpires that $\psi_\e$ needs to satisfy some solvability condition which is equivalent to the zero-determinant condition (\ref{2.71}) below. As we show in the next theorem, this condition is intrinsic for the limits of finite bending energy sequences in the case $h\ll \e^2$. This is due to the fact that ${\rm II} +\nabla_y^2 \psi_\e$ approximates the matrix of the second fundamental form of the mollified mid-surface $\widetilde u_h$ of the plate (defined in Section \ref{compactness}) on $\e$-scale. It appears that $\widetilde u_h$ has to be almost isometric in the case $h\ll \e^2$ for the expression $(\nabla' \widetilde u_h)^\top \nabla' \widetilde u_h - {\mathcal I}_2$ to be of order $o(h).$ Noting that  the matrix of the second fundamental form of an isometric surface has zero determinant suggests the condition 
${\rm det}({\rm II} +\nabla_y^2 \psi)=0,$ where $\psi$ is the two-scale limit of the sequence $\psi_\varepsilon.$  

\begin{theorem}[Zero determinant condition for the two-scale limit isometric surface.]
\label{isometryconstraint}
Let $u_h$ be a finite bending energy sequence, $\widetilde{u}_h$ be the mollified mid-surface of the plate defined by (\ref{2.28}) and $\omega'\Subset\omega$. Assume that $h\ll \e^2$. Then the weak two-scale limit $\nabla'^2 u + \nabla_y^2 \varphi$ of $\nabla'^2 \widetilde{u}_h$ in $L^2(\omega'\times Y)$, $u\in H^2_{\rm iso}(\omega')$, $\varphi\in L^2(\omega';H^2_{\rm per}(Y))$, satisfies the ``zero determinant'' condition
\begin{equation}\label{2.71}
 \det({\rm II} +\nabla_y^2 \psi) = ({\rm II}_{11}+\partial_{y_1y_1}^2 \psi) ({\rm II}_{22}+\partial_{y_2y_2}^2 \psi) -  ({\rm II}_{12}+\partial_{y_1y_2}^2 \psi)^2 = 0,
\end{equation}
where $\psi: = - \varphi \cdot n$.
\end{theorem}
\begin{proof}
The proof is inspired by the general idea of the proof of the identity $\det {\rm II} =0$ for the $H^2$ isometric surfaces (see \cite{Pakzad}). In our setting, however, the mollified mid-surface $\widetilde u_h$  of the plate is not exactly isometric.  This, together with the two-scale nature of our problem, requires in-depth analysis on the microscale. 

The following identity is easily checked by a straightforward calculation (recall that $\widetilde{u}_h$ is a $C^\infty$ function):
\begin{equation*}
 \det \nabla'^2 \widetilde{u}_h = \partial^2_{12} (\partial_1 \widetilde{u}_h \cdot \partial_2 \widetilde{u}_h) - \frac{1}{2} \big(\partial^2_{22} (\partial_1 \widetilde{u}_h)^2 +\partial^2_{11} (\partial_2 \widetilde{u}_h)^2\big).
\end{equation*}
Note that we can consider the third order tensor $\nabla^2 \widetilde{u}_h$ as a $2\times 2$ matrix whose elements $\partial_{ij} \widetilde{u}_h$ are  vector-valued functions. Thus when we take its determinant we understand the multiplication of entries as the scalar product of vectors. We will use this convention henceforth. We can rewrite the above identity applying the unfolding operator:
\begin{equation}\label{2.73}
 \mathcal T_\e(\det \nabla^2 \widetilde{u}_h) = \e^{-2}\Big(\partial^2_{y_1y_2} \mathcal T_\e (\partial_1 \widetilde{u}_h \cdot \partial_2 \widetilde{u}_h) - \frac{1}{2} \big(\partial^2_{y_2^2} \mathcal T_\e (\partial_1 \widetilde{u}_h)^2 +\partial^2_{y_1^2} \mathcal T_\e (\partial_2 \widetilde{u}_h)^2\big)\Big).
\end{equation}
We consider the left-hand side of the last identity as a scalar product of two ``vectors'' $f_\e:=(\mathcal T_\e(\partial_{11}^2 \widetilde{u}_h), \mathcal T_\e(\partial_{12}^2 \widetilde{u}_h))^\top$ and $g_\e:=(\mathcal T_\e(\partial_{22}^2 \widetilde{u}_h), - \mathcal T_\e(\partial_{12}^2 \widetilde{u}_h))^\top$.  Since ${\rm curl_y} f_\e= 0$ and $\div_y g_\e=0$, we can apply the compensated compactness theorem, see {\it e.g.} \cite{JKO}, which implies that the left-hand side of (\ref{2.73}) *-weak converges to the scalar product of their weak limits, that is
\begin{equation*}
 \lim_{h\to 0} \int_{\omega'} \int_Y \mathcal T_\e(\det \nabla'^2 \widetilde{u}_h) \eta dy dx' =  \int_{\omega'} \int_Y\det (\nabla'^2 u + \nabla_y^2 \varphi) \eta dy dx'
\end{equation*}
for any $\eta \in C_0^\infty (\omega'\times Y)$. Note that $x'$ plays the role of a parameter, and does not affect our argument. Let us now consider the right-hand side of (\ref{2.73}). Due to Proposition \ref{p3.3} we have $\partial_i \widetilde{u}_h = R_{h,i} + O(h)$, $h\to 0$, $i=1,2$, where $O(h)$ is understood in terms of the $L^2$-norm. In particular, $\partial_1 \widetilde{u}_h \cdot \partial_2 \widetilde{u}_h = O(h)$, $(\partial_i \widetilde{u}_h)^2=1+O(h)$. Thus 
\begin{equation}\label{2.75}
\begin{gathered}
 \e^{-2}\Big(\partial^2_{y_1y_2} \mathcal T_\e (\partial_1 \widetilde{u}_h \cdot \partial_2 \widetilde{u}_h) - \frac{1}{2} \big(\partial^2_{y_2^2} \mathcal T_\e (\partial_1 \widetilde{u}_h)^2 +\partial^2_{y_1^2} \mathcal T_\e (\partial_2 \widetilde{u}_h)^2\big)\Big) 
 \\
 =\e^{-2}\Big(\partial^2_{y_1y_2} O(h) - \frac{1}{2} \big(\partial^2_{y_2^2} (1+O(h)) +\partial^2_{y_1^2}  (1+O(h))\big)\Big) 
 \\
 =\partial^2_{y_1y_2} O(h/\e^2)  + \partial^2_{y_2^2} O(h/\e^2) +\partial^2_{y_1^2}  O(h/\e^2).
\end{gathered}
\end{equation}
Multiplying the last expression in (\ref{2.75}) by an arbitrary function $\eta \in C_0^\infty (\omega'\times Y)$ and integrating by parts twice we see that it converges *-weakly to zero given that $h/\e^2 = o(1)$, {\it i.e.}
\begin{equation*}
 \int_{\omega'} \int_Y\det\bigl(\nabla'^2 u(x') + \nabla_y^2 \varphi(x',y)\bigr) \eta dy dx' = 0.
\end{equation*}
This implies
\begin{equation}\label{2.77}
\det\bigl(\nabla'^2 u(x') + \nabla_y^2 \varphi(x',y)\bigr) = 0.
\end{equation}

At the last step of our proof we need one additional property of the two-scale limit $\varphi(x',y)$. We know that the second-order derivatives of $u(x')$ are parallel to the vector $n(x'):$ one has 
\[
\partial_{ij}u\cdot R_k=\partial_i R_j\cdot R_k=\partial_jR_i\cdot R_k=\partial_i(R_j\cdot R_k)-R_j\cdot\partial_kR_i=-R_j\cdot\partial_kR_i=0,\ \ \ i,j,k\in\{1,2\},
\]
since at least two of the indices $i,j,k$ coincide. Notice further that  
\begin{equation}\label{103}
\begin{gathered}
  -\partial_{ij}u\cdot n=\partial_{1j}u\cdot(\partial_2u\wedge\partial_iu)+\partial_{2j}u\cdot(\partial_iu\wedge\partial_1u)
=\partial_iu\cdot(\partial_{1j}u\wedge\partial_2u)+\partial_iu\cdot(\partial_{1}u\wedge\partial_{2j}u)
\\
=\partial_iu\cdot\partial_j(\partial_{1}u\wedge\partial_2u)=\partial_iu\cdot\partial_j n={\rm II}_{ij},\ \ \ \ \ i,j\in\{1,2\},
\end{gathered}
\end{equation}
hence $\partial_{ij} u = - {\rm II}_{ij} n$. In particular, one has $\det \nabla'^2 u = -\det {\rm II}$. In order to derive (\ref{2.71}) from (\ref{2.77}) we need a similar property for $\varphi$ to be held. In fact, we require a much weaker property of $\partial_{y_i} \varphi$ being orthogonal to $R_i$ for $i=1,2$. We establish it with the help of the following statement.

\begin{proposition}\label{l2.5}
Let 
 \begin{equation*}
  z_\e(x,y):= \e^{-1}\Big(\e^{-1}\big(\mathcal T_\e(\widetilde{u}_h)-\mathcal M_Y(\mathcal T_\e(\widetilde{u}_h))\big) - \mathcal M_Y(\mathcal T_\e(\nabla' \widetilde{u}_h))\,y^c\Big),
 \end{equation*}
(cf. (\ref{2.56})). Then
 \begin{equation}\label{2.79a}
  \|\partial_{y_i} z_\e \cdot \mathcal T_\e (R_{h,i})\|_{L^1(\omega'\times Y)} \to 0 \mbox{ as } h\to 0,\,i=1,2.
 \end{equation}
\end{proposition}
In the proof of this proposition we will need the following simple inequality, which is a sort of Poincar\'e inequality.

\begin{lemma}\label{l2.6}
Let $f\in H^1(Y)$. Then 
 \begin{equation*}
\int\limits_{Y\times Y} |f(y) - f(\zeta)|^2 dy d\zeta\leq 2 \int\limits_{Y} |\nabla f(y)|^2 dy.
 \end{equation*}
\end{lemma}
\begin{proof}
 Since the set of $C^\infty(Y)$ functions is dense in $H^1(Y)$ we only need to prove the lemma for  $f\in C^\infty(Y)$. First we write down a pointwise estimate for $|f(y)-f(\zeta)|$.
 \begin{equation*}
 \begin{gathered}
    |f(y)-f(\zeta)| = \left|\int_\zeta^y \nabla f \cdot dl\right| = \left|\int_{\zeta_1}^{y_1} \partial_1 f(t,\zeta_2)dt + \int_{\zeta_2}^{y_2} \partial_2 f(\zeta_1,t) dt\right|\leq
    \\
\leq \int_0^1 (|\nabla f(t,\zeta_2)|+ |\nabla f(\zeta_1,t)| ) dt.
 \end{gathered}
 \end{equation*}
We finish the proof with
 \begin{equation*}
 \begin{gathered}
\int\limits_{Y\times Y} |f(y) - f(\zeta)|^2 dy d\zeta\leq   \int\limits_{Y\times Y} \int_0^1 (|\nabla f(t,\zeta_2)|^2 + |\nabla f(\zeta_1,t)|^2 ) dt dy d\zeta = 2  \int\limits_{Y} |\nabla f(y)|^2 dy.
 \end{gathered}
 \end{equation*}
\end{proof}

\begin{proof}[Proposition \ref{l2.5}]
Taking the gradient of the function $z_\e$ and using the properties of the unfolding operator we have 
\begin{equation*}
 \nabla_y z_\e = \e^{-1}\big(\mathcal T_\e(\nabla' \widetilde{u}_h)- \mathcal M_Y(\mathcal T_\e(\nabla' \widetilde{u}_h))\big).
\end{equation*}
Let us write the $L^2$-norm of the right-hand side (dropping the coefficient $\e^{-1}$) in a special form:
\begin{equation}\label{2.85}
 \begin{gathered}
  \|\mathcal T_\e(\nabla' \widetilde{u}_h)- \mathcal M_Y(\mathcal T_\e(\nabla' \widetilde{u}_h))\|_{L^2(\omega'\times Y)}^2 = \int\limits_{\omega'\times Y} \left|\mathcal T_\e(\nabla' \widetilde{u}_h)- \int\limits_Y \mathcal T_\e(\nabla' \widetilde{u}_h)d\zeta\right|^2 dy dx = 
\\
=\int\limits_{\omega'\times Y} \left|\int\limits_Y (\mathcal T_\e(\nabla' \widetilde{u}_h)(x,y)- \mathcal T_\e(\nabla' \widetilde{u}_h)(x,\zeta))d\zeta\right|^2 dy dx \leq
\\
\leq \int\limits_{\omega'\times Y\times Y} \left| \mathcal T_\e(\nabla' \widetilde{u}_h)(x,y)- \mathcal T_\e(\nabla' \widetilde{u}_h)(x,\zeta)\right|^2 d\zeta dy dx
 \end{gathered}
\end{equation}
Applying Lemma \ref{l2.6} to $\mathcal T_\e(\nabla' \widetilde{u}_h)$ we have
\begin{equation}\label{2.86}
 \begin{gathered}
\int\limits_{\omega'\times Y\times Y} \left| \mathcal T_\e(\nabla' \widetilde{u}_h)(x,y)- \mathcal T_\e(\nabla' \widetilde{u}_h)(x,\zeta)\right|^2 d\zeta dy dx \leq 2 \| \nabla_y\mathcal T_\e(\nabla' \widetilde{u}_h)\|_{L^2(\omega'\times Y)}^2 =
\\
= 2 \e^2\| \mathcal T_\e(\nabla'^2 \widetilde{u}_h)\|_{L^2(\omega'\times Y)}^2 \leq C\e^2,
 \end{gathered}
\end{equation}
where in the last inequality we used the property of boundedness of the second gradient of $\widetilde{u}_h$, see Theorem \ref{l3.2}. Due to Proposition \ref{p3.3} we can replace $\mathcal T_\e(\nabla' \widetilde{u}_h)$ in the last estimate by $\mathcal T_\e(\widehat R_h)$:
\begin{equation}\label{2.87}
 \begin{gathered}
\int\limits_{\omega'\times Y\times Y} \left| \mathcal T_\e(\widehat R_h)(x,y)- \mathcal T_\e(\widehat R_h)(x,\zeta)\right|^2 d\zeta dy dx \leq C\e^2.
 \end{gathered}
\end{equation}

Similarly to (\ref{2.85}) we can write the expression in (\ref{2.79a}) as
\begin{equation*}
\int\limits_{\omega'\times Y\times Y}\e^{-1} \big(\mathcal T_\e(\partial_i \widetilde{u}_h)(x,y)- \mathcal T_\e(\partial_i \widetilde{u}_h)(x,\zeta)\big) \cdot \mathcal T_\e (R_{h,i})(x,y) d\zeta dy dx.
\end{equation*}
As follows from Proposition \ref{p3.3}, in order to prove the statement it is enough to show that 
\begin{equation*}
\e^{-1} \int\limits_{\omega'\times Y\times Y}\big(\mathcal T_\e(R_{h,i})(x,y)- \mathcal T_\e(R_{h,i})(x,\zeta)\big) \cdot \mathcal T_\e (R_{h,i})(x,y) d\zeta dy dx \to 0,\,\,i=1,2.
\end{equation*}
Notice that the integrand in the latter is non-negative everywhere. The expression under the sign of integral has the form $(a-b)\cdot a$ at every point, where $a$ and $b$ are unit vectors. We use the following elementary identity, $2 (a-b)\cdot a = |a-b|^2$, together with (\ref{2.87}) to conclude the proof of the proposition:
 \begin{equation*}
 \begin{gathered}
\e^{-1} \int\limits_{\omega'\times Y\times Y}\big(\mathcal T_\e(R_{h,i})(x,y)- \mathcal T_\e(R_{h,i})(x,\zeta)\big) \cdot \mathcal T_\e (R_{h,i})(x,y) d\zeta dy dx  
\\
\quad\quad\quad\quad\quad\quad = \frac{1}{2} \e^{-1} \int\limits_{\omega'\times Y\times Y}\big|\mathcal T_\e(R_{h,i})(x,y)- \mathcal T_\e(R_{h,i})(x,\zeta)\big|^2 d\zeta dy dx \leq C\e.
\end{gathered}
\end{equation*}
 \end{proof}
 
Theorem \ref{th2.1} implies that the sequence  $z_\e$ converges weakly in $L^2\bigl(\omega'; H^2(Y)\bigr)$ to the function ({\it cf.} (\ref{triangle})) 
\begin{equation*}
 z=\frac{1}{2}(y_1^c)^2 \partial_1^2 u +  y_1^c y_2^c \partial_1\partial_2 u + \frac{1}{2}(y_2^c)^2 \partial_2^2 u + \varphi.
\end{equation*}
Since $\mathcal T_\e(R_h)$ converges to $R$ strongly in $L^2(\omega'\times Y)$, we have
\begin{equation*}
 \partial_{y_i} z_\e \cdot \mathcal T_\e(R_{h,i}) \rightharpoonup \partial_{y_i} z\cdot R_i,
\end{equation*}
*-weakly, {\it i.e.} with test functions taken from $C^\infty_0(\omega'\times Y)$. (In fact, with some effort it is possible to show weak convergence in $L^2(\omega'\times Y)$, however, *-weak convergence is sufficient for our purposes.) Further, since the second-order derivatives of $u$ are parallel to the vector $b$, we have
\begin{equation*}
\partial_{y_i} z\cdot R_i = \partial_{y_i} \varphi\cdot R_i.
\end{equation*}
Lastly, it follows from Proposition \ref{l2.5} that the *-weak limit of $\partial_{y_i} z_\e \cdot \mathcal T_\e(R_{h,i})$ equals zero, hence
\begin{equation*}
\partial_{y_i} \varphi\cdot R_i = 0, \,\, i=1,2.
\end{equation*}
The above orthogonality property implies that the derivatives of $\varphi$ can be represented in the form $\partial_{y_1} \varphi = (\partial_{y_1} \varphi\cdot R_2) R_2+(\partial_{y_1} \varphi\cdot n)n$, $\partial_{y_2} \varphi = (\partial_{y_2} \varphi\cdot R_1) R_1+(\partial_{y_2} \varphi\cdot n)n$. Differentiating further we get the following identities for the second derivatives of $\varphi$ (in particular we see that the mixed derivative is orthogonal to both $R_1$ and $R_2$):
\begin{equation*}
\begin{gathered}
\partial_{y_1}^2 \varphi = (\partial_{y_1}^2 \varphi \cdot R_2) R_2 + (\partial_{y_1}^2 \varphi \cdot n) n,
\\
\partial_{y_2}^2 \varphi = (\partial_{y_2}^2 \varphi \cdot R_1) R_1 + (\partial_{y_2}^2 \varphi \cdot n)n,
\\
\partial_{y_1}\partial_{y_2} \varphi = (\partial_{y_1}\partial_{y_2} \varphi \cdot n)n.
\end{gathered}
\end{equation*}
Substituting the latter and $\nabla'^2 u = -{\rm II}\, n$ into (\ref{2.77}) and recalling that $\psi = - \varphi \cdot n$ we obtain  (\ref{2.71}).
\end{proof}

\section{The structure of isometric immersions}\label{isometricimmersions} 
The idea of the construction of a recovery sequence for the supercritical regime $h\ll \e^2$, which is performed in the next section, relies on the properties of isometric surfaces from $H^2_{\rm iso}(\omega)$. There has recently been a number of works devoted to this topic, see \cite{Hornung2008}, \cite{Hornung2011}, 
\cite{MuellerPakzad}, \cite{Pakzad}. The results of Section \ref{s5} (as discussed at the beginning of Section \ref{extremeenergy}) imply a discontinuous dependence of the homogenised energy density function on the direction of the bending. This observation drastically reduces the number of manipulations one can perform with a surface in order to construct a recovery sequence. In particular, we can not use the result on approximation of $H^2_{\rm iso}(\omega)$ surfaces by infinitely smooth ones because this approximation does not preserve the orientation of the bending directions.  In our paper we will use mostly the results of \cite{Pakzad}, which we review below. We then give an example of an isometric surface for which the set of rational directions is irregular ({\it cf.} Theorem \ref{th2.8}), and prove a statement on regularisation of isometric surfaces required for the construction of a recovery sequence.

\subsection{\bf Auxiliary observations}

For any $u\in H^2_{\rm iso}(\omega)$ domain $\omega$ can be partitioned into ``bodies'' and ``arms'', according to the following definition.
\begin{definition}\label{def:leadingcurve}

A {\bf body} is a maximal subdomain of $\omega$ on which $u$ is affine and whose boundary contains at least three segments inside 
$\omega$. 

A differentiable curve $\Gamma:[0,l]\to \omega$ is referred to as a {\bf leading curve} if $u$ is affine when restricted to the lines normal 
to $\Gamma\bigl([0,l]\bigr),$ which we refer to as {\bf leading segments}. In the approximation procedure below we also use the term ``leading curve'' for the curve $\gamma:=(u\circ\Gamma)$ on the 
surface.


For a given leading curve $\Gamma,$ {\bf generatrices} are the lines tangent to the surface $u$ at the points of $\gamma$ and orthogonal to the latter.

An {\bf arm}  $\omega[\Gamma]$ is a subdomain of $\omega$ covered by leading segments corresponding to a leading curve $\Gamma.$ (See Remark \ref{armrepres} below.)  

\end{definition}

\begin{remark}
Note that a maximal subdomain of $\omega$ on which $u$ is affine and hose boundary contains at most two segments inside 
$\omega$ is always an arm or part of an arm.
\end{remark}

\begin{remark}
\label{armrepres}
Suppose that $\Gamma$ is a leading curve parametrised by its arclength, and denote by $T:=\Gamma'$ the tangent vector 
and by $N:=(-T_2,T_1)^\top$ the unit vector orthogonal to $\Gamma$. 
Then the corresponding arm is given by 
\begin{equation}\label{6.134}
 \omega[\Gamma] = \bigl\{\Gamma(t)+sN(t): s\in \mathbb{R}, t\in[0,l]\bigr\}\cap\omega.
\end{equation}
For each $t\in[0,l]$, the set $\bigl\{\Gamma(t)+sN(t): s\in \mathbb{R}\bigr\}\cap\omega$ is the corresponding leading segment. 
\end{remark}

Since leading curves are not uniquely defined, it is more reasonable to define the vector field $N$ on the whole $\omega[\Gamma]$ as a continuous vector field of unit vectors parallel to leading segments, and the vector field $T$ as the continuous unit vector field orthogonal to $N$. In particular, we have $T=\Gamma'$ for any leading curve $\Gamma$. In what follows we consider $N$ and $T$ either as vector fields $N(x')$, $T(x')$ or the normal $N(t)$ and the tangent $T(t)$ of the curve $\Gamma(t)$, depending on the context. We will say that the value of $T$ (equivalently, $N$) is {\bf rational} if the ratio of its components is rational. We will also say in this situation that the corresponding leading segment is rational.

\begin{remark}
It can be shown that: 

a) For any given body or arm the part of its boundary that is situated in the interior of $\omega$ is a union of straight segments whose vertices belong to $\partial\omega;$

b) As follows from \cite{Pakzad}, all leading curves are in the class $C^{1,1}_{\rm loc}.$ A combination of this observation with the result of \cite{MuellerPakzad} establishing that $H^2_{\rm iso}\subset C^1(\omega)$ implies the fact that leading curves $\gamma$ on the surface 
$u$ are in the class $C^1.$
\end{remark}

Since the leading segments do not intersect in $\omega$, the vector field $N$ (and, hence, $T$) is locally Lipschitz in $\omega$. In is not difficult to prove, using a geometric argument, the following lemma.

\begin{lemma}\label{l6.4}
 The local Lipschitz constant
 \begin{equation*}
  C_L(x_0'):= \limsup_{x'\to x_0'} \frac{|N(x')-N(x_0')|}{|x'-x_0'|}=\limsup_{x'\to x_0'} \frac{|T(x')-T(x_0')|}{|x'-x_0'|}, \qquad x_0'\in \omega
 \end{equation*}
satisfies the estimate
\begin{equation*}
 C_L(x_0') \leq \dist^{-1}(x_0',\,\partial \omega).
\end{equation*}
\end{lemma}
This lemma implies, in particular, local Lipschitz continuous differentiability of leading curves.

Note that for any leading curve $\Gamma,$ unit tangent vectors of $\gamma$  
are also tangent to the surface $u: {\mathbb R}^2\to{\mathbb R}^3.$ Therefore, the surface $u$ is reconstructed on each arm from the knowledge of the corresponding leading curve
and the non-vanishing principal curvature. 
Indeed,  denote by $\kappa_\Gamma$ the curvature of $\Gamma,$ so that $\Gamma''=T'=\kappa_\Gamma N$. For a given $\gamma,$ we define a Darboux frame $(\tau, \nu, n)$ by the tangent $\tau = \gamma'$ to the curve $\gamma,$ the unit vector $\nu =(\nabla'u)N$ tangent to $u$ and orthogonal to $\tau$, and $n=\tau \wedge \nu$. This Darboux frame satisfies the system of equations (see {\it e.g.} \cite[p. 277]{Spivak}, 
\cite{Pakzad})
\begin{equation}\label{2.108}
\left(\begin{array}{c}
\tau'
\\
\nu'
\\
n'
\end{array}\right)= \left(\begin{array}{ccc}
0 & \kappa_\Gamma & \kappa
\\
-\kappa_\Gamma & 0 & 0
\\
-\kappa & 0 & 0
\end{array}\right)\left(\begin{array}{c}
\tau
\\
\nu
\\
n
\end{array}\right)            
\end{equation}
which has a unique solution subject to appropriate initial values. Given $\gamma(t)$ and $\nu(t)$ we can write $u|_{\omega[\Gamma]}$ as
\begin{equation}\label{6.136}
u=\gamma(t)+s\nu(t),
\end{equation}
where $s$ parametrises the generatrix corresponding to the value $t$ on the curve $\Gamma.$
\begin{remark}
 There is freedom in the choice of the sign of $\kappa_\Gamma$ and the direction of the vector field $N$. We make this choice in such a way 
 that the normal $n$ in the above Darboux frame coincides with the normal $n$ defined as $\partial_1 u\wedge\partial_2 u$ in the discussion preceding the current section.
\end{remark}

Let us write down the formulae for for the gradients $\nabla' u,\,\nabla'^2 u$ and the matrix ${\rm II}$ in terms of the leading directions and curvatures. Since $x'$ is related to $(t, s)$ via $x'(t,s)=\Gamma(t)+sN(t)$. We have
\begin{equation*}
 \frac{\partial (x_1,x_2)}{\partial(t,s)} = (\Gamma' + s N', N) = 
\left(
 \begin{array}{cc}
  {T}_1(1-s \kappa_\Gamma) & -{T}_2
  \\
  \tg_2  (1-s \kappa_\Gamma) & \tg_1
 \end{array}
\right),
\end{equation*}
and therefore
\begin{equation*}
 \frac{\partial(t,s)}{\partial (x_1,x_2)}= \left(
 \begin{array}{cc}
  {T}_1(1-s \kappa_\Gamma)^{-1} & \tg_2  (1-s \kappa_\Gamma)^{-1}
  \\
   -{T}_2 & \tg_1
 \end{array}
\right).
\end{equation*}
The above expression is well defined since $1-s \kappa_\Gamma$ is positive everywhere on $\o[\Gamma]$, {\it cf.} Lemma \ref{l6.4}. Then we can write the gradient and the second gradient of $u$ in terms of the geometry of the curves $\gamma$ and $\Gamma$:
\begin{equation}\label{145}
 \begin{gathered}
  \nabla' u = \big(T_1 \tau - T_2 \nu,\,T_2 \tau + T_1 \nu\big),
  \\
  \nabla'^2 u = \frac{\kappa}{1-s \kappa_\Gamma}\left(
 \begin{array}{cc}
  T_1^2\, n & T_1 T_2\, n
  \\
  T_1 T_2\, n &  T_2^2\, n
 \end{array}
\right).
 \end{gathered}
\end{equation}
Notice that ({\it cf.} (\ref{103}))
\begin{equation}\label{147}
 {\rm II} = -\frac{\kappa}{1-s \kappa_\Gamma}\left(
 \begin{array}{cc}
  T_1^2 &  T_1 T_2
  \\
  T_1 T_2 & T_2^2
 \end{array}
\right).
\end{equation}
Let us introduce the following notation
\begin{equation}\label{53}
 \widetilde{\rm II} = \left(
 \begin{array}{cc}
  T_1^2 &  T_1 T_2
  \\
  T_1 T_2 & T_2^2
 \end{array}
\right).
\end{equation}
As will be shown in the next section, the homogenised stored energy density $Q_{\rm hom}^{\rm sc}$ in the supercritical regime is in general a discontinuous function of its argument ${\rm II}$. More precisely, it is continuous with respect to the coefficient $-\kappa(1-s \kappa_\Gamma)^{-1}$, but discontinuous with respect to $\widetilde{\rm II}$ (the direction $T$). This suggests that one can modify the curvature $\kappa$ in a continuous way ({\it i.e.} without changing much the homogenised energy), however, a perturbation of the vector field $T$ may lead to uncontrollable changes in the homogenised energy.


\subsection{\bf Regularisation of irregular surfaces.}

The results of \cite{Pakzad} and \cite{MuellerPakzad} assert some regularity of isometric immersions $u\in H^2_{\rm iso}(\omega)$, namely, $u\in C^1(\omega)$ and its associated leading curves in $\omega$ have locally Lipschitz derivative. Nevertheless, isometric immersions can be quite irregular for the purpose of this paper. To illustrate this, we  give an example of a surface for which the set of ``rational'' directions for the corresponding leading curve is the union of a fat Cantor set and a set of zero measure.

{\bf Example.}
Let us recall the construction of a fat Cantor set. At the first step we remove from the middle of $[0,1]$ an open interval $B_1$ of length $1/4$ and denote the remaining set $A_1.$ At the next step from the middle of each the two subintervals of $A_1$ we remove open intervals of length $1/16$ each, the remaining set is denoted by $A_2$ and its complement by $B_2$ and so on: 
at the $n$-th step we remove $2^{n-1}$ disjoint intervals of length $1/2^{2n},$ whose union is a set of measure $1/2^{n+1}.$
Thus we have $\lambda_1(A) = \lambda_1(B) = \frac{1}{2}$, where $A:=\cap_n A_n$, $B:= \cup_n B_n$. By $\lambda_d$ we denote the Lebesgue measure in $\mathbb R^d$. The Fat Cantor Set $A$ is closed and is nowhere dense ({\it i.e.} contains no interior points). For each $n$ we define a leading curve via its unit tangent vector $T_n(t)$ given as follows. Let $\phi$ be a function that is continuously differentiable on the interval $[0,1],$ has zero mean and is such that $\phi(0)=\phi(1)=\phi'(0)=\phi'(1)=0$, {\it e.g.} 
$\phi(t)=\bigl(\sin(2\pi t)\bigr)^3,$ and set $A_0 = [0,1]$, $B_0 = \varnothing$, $T_0:=(1,0)^\top$. For each $n=1,2,\ldots,$ we define $T_n:= T_{n-1}$ on the set $[0,1]\setminus(B_n\setminus B_{n-1}),$ and on each component interval $I_{n,k}$ of the set $B_n\setminus B_{n-1}$ we define 
\[
T_n:= \Bigl(\sqrt{1 - \phi_{n,k}^2},\, \phi_{n,k}\Bigr)^\top,\ \ \ \ k=1,2,..., 
\]
where $\phi_{n,k} = \beta_n\phi(2^{2n} t - t_{n,k})$ is the scaled (by a coefficient $\beta_n,$ $\vert\beta_n\vert<1,$ and the length of the interval $1/2^{2n}$) and appropriately translated (by the left end $t_{n,k}$ of the interval $I_{n,k}$) version of $\phi$, which ``fits'' into this interval. We then integrate $T_n$ to obtain the curves $\Gamma_n$:
\begin{equation*}
 \Gamma_n(t):= \int\limits_0^t T_n, \,\, t\in[0,1].
\end{equation*}
Note that at each step of this procedure we ``redefine'' the vector $T_n$ only on the set $B_n\setminus B_{n-1}.$

We take $\beta_n=\beta/2^{2n}$, where $0<\beta<4\bigl(\max\vert\phi\vert\bigr)^{-1}$ is a constant that will  be specified later.
So far we have a sequence of curves $\Gamma_n$ with a bounded second derivative\footnote{A direct calculation shows that 
 \[
 \vert\Gamma_n''\vert\le\vert\phi'_{n,k}\vert(1-\phi_{n,k}^2)^{-1/2}=2^{2n}\beta_n\vert\phi'\vert(1-\beta_n^2\phi^2)^{-1/2}, 
 \]
 which {\it e.g.} for the choice $\phi(t)=\bigl(\sin(2\pi t)\bigr)^3,$ $t\in[0,1],$ and $\beta\le2\sqrt{2}$ gives the value $C=4\pi\sqrt{2/3}.$}: $\Gamma_n''\le C\beta.$ It has a subsequence that converges strongly in $H^1(0,1)$ to a curve $\Gamma\in H^2(0,1)$ with $|\Gamma''|\leq C\beta$. In particular, one has $\Gamma' = (1,0)^\top$ on the fat Cantor set defined above, and the subset of $B$ on which $\Gamma'$ takes rational directions has zero measure. Let $\omega:=[0,l]\times[0,1]$, where $l<1$ is defined by the right end of $\Gamma$, namely it is the first coordinate of the point $\Gamma(1)$. If we choose $\beta$ such that $|\Gamma''|\leq C\beta \leq ({\rm diam}(\omega))^{-1}$, the lines normal to $\Gamma$ do not intersect within 
 $\omega.$ (Clearly, one has ${\rm diam}(\omega)\le\sqrt{2},$ so the choice 
 $\beta\le1/(C\sqrt{2})$ satisfies the required condition.) Lastly, we can construct an isometric surface $u$ via procedure described above (see formulae (\ref{2.108}) and (\ref{6.136})) with some ``reasonable'' curvature $\kappa$, {\it e.g.} $\kappa\equiv 1$. For thus constructed $u$ the domain $\omega$ is the only arm and $\Gamma$ is its leading curve in the sense of Definition \ref{def:leadingcurve}.
 
  \bigskip
 
 The above example illustrates that for a leading curve $\Gamma(t)$ the set on which $T(t)$ has rational directions can not in general be represented as a union of intervals (up to a null set). In what follows we argue that one can replace an arbitrary isometric surface $u$ by a ``sufficiently regular'' isometric surface, in order to obtain an explicit construction of a recovery sequence. Note that one can not simply approximate $u$ by an infinitely smooth isometric surface, since $Q_{\rm hom}^{\rm sc},$ in general, is an everywhere discontinuous function of its argument. Thus the matrix $\widetilde{\rm II}$, see (\ref{53}), of the replacement surface should be different from the one of $u$ only on a relatively small set. First we need to recall some properties of isometric surfaces. We reformulate and complement some results of \cite{Pakzad} (Lemmas 3.8 and 3.9) in the following statement.
 \begin{theorem}\label{th6.5}
  Let $u \in H^2_{\rm iso}(\omega)$. For any $\delta>0$ there exists $v \in H^2_{\rm iso}(\omega)$ such that 
\[
 \lambda_2(\{x':\, u(x')\neq v(x')\})+\|u-v\|_{H^2(\omega)}\leq \delta,
\] 
the surface $v$ has a finite number of bodies and arms, each maximal connected region with non-zero mean curvature lies between two regions where $v$ is affine, the corresponding leading curves are Lipschitz differentiable and do not have common points with $\partial \omega$.
 \end{theorem}
The above properties of $v$ are partly derived from the proofs of Lemmas 3.8 and 3.9 in \cite{Pakzad}. Next, we briefly elucidate them. Every maximal connected region (we denote it by $\Sigma$) where $u$ has non-zero mean curvature is attached to either one or two regions where $u$ is affine. If $\Sigma$ is situated between two regions where $u$ is affine, then the lengths of the leading segments that cover $\Sigma$ are bounded below by a positive constant. Lemma \ref{l6.4} implies that on any closed subset of $\omega$ the vector field $T$ is  Lipschitz continuous. In particular, we can cover $\Sigma$ by a finite number of leading curves so that they do not have common points with the boundary of $\omega$, and hence have Lipschitz continuous derivative. If $\Sigma$ has just one ``affine'' region attached to it or $\Sigma = \omega$, then it may happen that the infimum of the length of leading segments covering $\Sigma$ is zero. Then one can not control the Lipschitz constant on 
the sequence of segments which lengths tend to zero (such sequence will converge to a point on the boundary of $\omega$). In this case we can replace $u$ in a small region near the boundary point where such degeneration of leading segments occurs by an affine function making sure that the new surface is still an element of $H^2_{\rm iso}(\omega)$. Then we get to the situation discussed previously.
 
As was mentioned at the beginning of this section we aim at the construction of recovery sequences for the homogenised functional $E_{\rm hom}^{\rm sc}$ in the supercritical case $h\ll\e^2$, given by (\ref{63}). Theorem \ref{th6.5} implies, in particular, that for $\delta>0$ there exists $v \in H^2_{\rm iso}(\omega)$ with a finite number of bodies and arms that depends on $\delta$ so that
 \begin{equation*}
 \bigl|E_{\rm hom}^{\rm sc}(u) -  E_{\rm hom}^{\rm sc}(v)\bigr|\leq\delta.
\end{equation*}

The next theorem ensures that we can replace $u \in H^2_{\rm iso}(\omega)$ with an isometric surface such that the set on which its leading segments are rational can be represented as a union of finite number of arms and a set of arbitrary small measure. 
\begin{theorem}\label{th6.6}
 Let $u \in H^2_{\rm iso}(\omega)$ be such that the subset $\omega_r(u)$ of $\omega$ covered by rational leading segments has positive measure. Then for each $\delta>0$ there exists a surface $u_\delta\in H^2_{\rm iso}(\omega)$ such that: 
 
 a) The surface $u_\delta$ has a finite number of bodies and arms, which we generically denote by $\omega[\tilde\Gamma];$ 
 
 b) The subset $\omega_r(u_\delta)$ of $\omega$ covered by rational leading segments can be represented as  
 \[
 \omega_r(u_\delta) = \Big(\bigcup_{\tilde\Gamma_r}\omega[\tilde\Gamma_r]\Big) \cup S_\delta,
 \]
where each arm $\omega[\tilde\Gamma_r]$ consists of parallel rational segments (in particular, each leading curve $\tilde\Gamma_r$ is a straight segment and $u(\omega[\tilde\Gamma_r])$ is cylindrical); 
 
 c) The set $S_\delta$ is measurable; 
 
 d) The estimate
 \begin{equation}\label{134}
 \lambda_2(S_\delta)+ \|u-u_\delta\|_{H^2(\omega)} + |E_{\rm hom}^{\rm sc}(u) -  E_{\rm hom}^{\rm sc}(u_\delta)|<\delta
\end{equation}
holds. 

Moreover, $u_\delta$ can be chosen in such a way that on each arm the corresponding curvature $\kappa$ ({\it cf.} (\ref{2.108}) and (\ref{147})) is infinitely smooth and ${\rm II}_\delta\in L^\infty(\omega)$, where the latter in the matrix of the second fundamental form of $u_\delta$.
\end{theorem}
\begin{remark}
The conditions that the curvature $\kappa$ is smooth and ${\rm II}_\delta\in L^\infty(\omega)$ is required for the construction of a recovery sequence. In particular, we take advantage of these facts when using the Taylor expansion formula for $W$, see (\ref{WTaylor}).
\end{remark}
\begin{proof}
In view of Theorem \ref{th6.5} we may assume without loss of generality that $u$ has a finite number of bodies and arms from the outset. Let us consider an arm $\omega[\Gamma]$ with the corresponding leading curve $\Gamma(t)$, $t\in [0,l]$, and assume that the subset $K= K(\Gamma)$ of $[0,l]$ on which the directions of $T=T(t)$ are rational has positive measure $\lambda_1(K)>0$. Since the set of rational directions is countable, we can order its elements in such a way that $\lambda_1(K_1)\geq \lambda_1(K_2)\geq \ldots$, where $K_i$ is the preimage of a particular rational direction 
$T_i,$ $i=1, 2, 3, ...,$ so that $K=\cup_i K_i$. We take an arbitrary $\eta>0$ and let $m=m(\eta)$ be such that for $ K_\eta:=\cup_{i=1}^m K_i$ we have $\lambda_1(K\setminus  K_\eta)\leq \eta$. Note that all $K_i$ are closed, in particular, the distance between any two of these sets is positive.
 This implies that we can remove a finite 
number of open intervals $\tilde L_i \subset [0,l]\setminus K_\eta,\,i=1\ldots,k,$ so that the remaining set $L_\eta:=[0,l]\setminus (\cup_{i=1}^k \tilde L_i)$ (we may also remove the points $0$ and $l$ from $L_\eta$ if they are isolated) possesses the following properties: 
\begin{enumerate}
  \item The estimate $\lambda_1(L_\eta\setminus K_\eta)<\eta$ holds.
  \item The set $L_\eta$ is the union of a finite number of mutually disjoint closed intervals $L_j$, $j=1,\ldots,n.$
  \item For each $j$ one has $L_j\cap  K_\eta = L_j \cap K_i$ for some $i\in\{1, 2, ..., k\}$ and both ends of the interval $L_j$ belong to $K_i$. 
\end{enumerate}

Let us consider a particular interval $L_j=[a_j,b_j]$. The part of $\omega[\Gamma]$ covered by the restriction of $\Gamma$ to $L_j$, which we denote by $\omega[\Gamma|_{L_j}]$, is contained between two parallel segments $\Gamma(a_j) + s N_i$ and $\Gamma(b_j) + s N_i$ corresponding to the rational direction $T_i$. In fact, loosely speaking, $\omega[\Gamma|_{L_j}]$ is mostly covered by the leading segments which are parallel to $N_i$. On the rest of the set, corresponding to the restriction $\Gamma|_{L_j\setminus K_i}$, where the leading segments are not parallel to $N_i$, we replace them with new leading segments that are parallel to $N_i$. Notice that the difference between $N(t)$ and $N_i$ on the set $L_j\setminus K_i$ is not greater than $C \eta$, where $C$ is the Lipschitz constant of $N(t)$. We can write 
\[
\omega[\Gamma|_{L_j}]=\{\Gamma(t)+sN_i,\, s\in{\mathbb R} \, t\in L_j\}\cap\omega.
\]
We carry out this procedure for each interval $L_j$. In this way we change the leading segments covering $\omega[\Gamma]$ only on a small set $\omega_\eta[\Gamma]$. We estimate the measure of this set as follows: 
 \begin{equation}\label{135}
\lambda_2\bigl(\omega_\eta[\Gamma]\bigr) \leq \lambda_1(L_\eta\setminus K_\eta) {\rm diam}(\omega) \leq \eta\, {\rm diam}(\omega).
   \end{equation} 
Let us denote by $N_\eta$ and $T_\eta$ the vector fields corresponding to the new set of leading segments. We have 
 \begin{equation}\label{6.135a}
 \begin{gathered}
 N_\eta = N,\, T_\eta = T \mbox{ in } \omega[\Gamma]\setminus\omega_\eta[\Gamma],
\\
\|N_\eta - N\|_{L^\infty(\omega_\eta[\Gamma])} \to 0,\,\|T_\eta - T\|_{L^\infty(\omega_\eta[\Gamma])}\to 0 \mbox{ as } \eta\to 0.
 \end{gathered}
 \end{equation}
We denote by $S_\eta^\Gamma$ the union of all rational leading segments that do not belong to the set $\bigcup_{j=1}^n \omega[\Gamma|_{L_j}]$  and estimate its measure as follows: 
\begin{equation}\label{137}
\lambda_2\bigl(S_\eta^\Gamma\bigr) \leq \lambda_1(K\setminus  K_\eta)  {\rm diam}(\omega)\leq \eta\, {\rm diam}(\omega).
\end{equation}
 
 For each arm we build a family of leading curves $\Gamma_\eta$ and isometric surfaces $\mathcal{U}_\eta$, which will be shown to converge to the surface $u$ uniformly as $\eta\to 0$. We assemble the approximating parts $\mathcal{U}_\eta$ on all arms and bodies in a continuous fashion. We then choose a sufficiently small value of $\eta$ so that the resulting surface denoted by $u_\delta$ satisfies the inequality (\ref{134}).
 
 We define $\Gamma_\eta$ as the maximal solution in $\omega[\Gamma]$ to the following ODE:
 \begin{equation}\label{6.135}
  \Gamma_\eta'(t) = T_\eta\bigl(\Gamma_\eta(t)\bigr),\quad \Gamma_\eta(0)=\Gamma(0),
 \end{equation}
 where $T_\eta$ is the field of unit vectors orthogonal to the modified set of leading segments of $\omega[\Gamma]$. Notice that $\Gamma_\eta$ uniformly tends to $\Gamma$ as $\eta\to0.$ In particular, for small enough values of $\eta,$ the new curve will cover the whole $\omega[\Gamma]$, {\it i.e.} one has $\omega[\Gamma] = \omega[\Gamma_\eta]$.  We denote the domain of $\Gamma_\eta$ by $[0,l_\eta]$, noting that in general $l_\eta\neq l$. Both $T$ an $T_\eta$ are Lipschitz continuous on any compact subset of $\omega$, the measure of the set $\omega_\eta[\Gamma]$ on which $T \neq T_\eta$ is proportional to $\eta$. This implies that on $\omega[\Gamma_\eta]$ one has a point-wise estimate $|T(t) - T_\eta(t)|\leq C \eta$, where $C$ is proportional to the Lipschitz constant. It follows that the solution of (\ref{6.135}) converges to $\Gamma$ uniformly as $\eta\to 0$. To each leading segment of $\omega[\Gamma]$ correspond unique $t\in [0,l]$ and $\hat t\in [0,l_\eta]$ such that $\Gamma(t)$ and $\Gamma_\eta(\hat t)$ belong to this segment. Let us denote by $\zeta_\eta$ the corresponding map $\zeta_\eta: t\mapsto \hat t$. It is not difficult to see that $\zeta_\eta$ is differentiable with $\zeta_\eta'$ converging to $1$ as $\eta \to 0$. A direct calculation yields the following expression for the directional derivative of $T(x')$ in the direction of $T(x')$:
 \begin{equation*}
  (\nabla' T)\, T= \frac{\kappa_\Gamma(t)}{1-s \kappa_\Gamma(t)} N(t),
 \end{equation*}
where $(t,s)$ is such that $x' = \Gamma(t) + sN(t)$. In particular,  on the set $\o[\Gamma]\setminus \o_\eta[\Gamma]$ we have
\begin{equation}\label{6.137}
\begin{gathered}
  \kappa_{\Gamma_\eta}\circ\zeta_\eta(t) =  \frac{\kappa_\Gamma(t)}{1-\big(\Gamma_\eta\circ\zeta_\eta(t) - \Gamma(t)\big)\cdot N(t)\,\kappa_\Gamma(t)},\quad \kappa_{\Gamma_\eta}\circ\zeta_\eta(t) \to \kappa_\Gamma(t) \mbox{ as }\eta\to 0.
\end{gathered}
\end{equation}

Now we can construct a new surface on $\omega[\Gamma_\eta]$. To this end we will use the equations analogous to (\ref{2.108}) and (\ref{6.136}). First, we need to adjust the non-zero principal curvature of the surface $\kappa$ to compensate the shift of the leading curve $\Gamma_\eta$ along the leading segments, so let
\begin{equation}\label{6.139a}
\kappa_\eta(t):= \frac{\kappa(t)}{1-\big(\Gamma_\eta\circ\zeta_\eta(t) - \Gamma(t)\big)\cdot N(t) \,\kappa_\Gamma(t)} .
\end{equation}
We define the leading curve on the surface by
\begin{equation*}
 \gamma_\eta(\hat t) := \int\limits_0^{\hat t} \tau_\eta, \quad \gamma_\eta(0) = \gamma(0),
\end{equation*}
with a unique Darboux frame satisfying the following system of equations:
\begin{equation}\label{6.139}
\left(\begin{array}{c}
\tau'_\eta(\hat t)
\\
\nu'_\eta(\hat t)
\\
n'_\eta(\hat t)
\end{array}\right)= \left(\begin{array}{ccc}
0 & \kappa_{\Gamma_\eta}(\hat t)\ \ \  & \kappa_\eta\circ\zeta_\eta^{-1}(\hat t)
\\
-\kappa_{\Gamma_\eta}(\hat t) & 0 & 0
\\
-\kappa_\eta\circ\zeta_\eta^{-1}(\hat t)\ \ \ & 0 & 0
\end{array}\right)\left(\begin{array}{c}
\tau_\eta(\hat t)
\\
\nu_\eta(\hat t)
\\
n_\eta(\hat t)
\end{array}\right),
\end{equation}
with the initial conditions $\bigl(\tau_\eta(0),\nu_\eta(0),n_\eta(0)\bigr)=\bigl(\tau(0),\nu(0),n(0)\bigr)$, where $\zeta_\eta^{-1}$ is the inverse of $\zeta_\eta$. Then 
the new surface $\mathcal U_\eta$ is defined by
\begin{equation*}
\mathcal U_\eta(\hat t, s)=\gamma_\eta(\hat t)+s\nu_\eta(\hat t),
\end{equation*}
for suitable values of $\hat t$ and $s$. Substituting $\hat t = \zeta_\eta(t)$ we see that $\bigl(\tau_\eta\circ\zeta_\eta(t),\nu_\eta\circ\zeta_\eta(t),n_\eta\circ\zeta_\eta(t)\bigr)$ is the solution to the system
\begin{equation}\label{6.141}
\frac{d}{dt}\left(\begin{array}{c}
\tau_\eta\circ\zeta_\eta(t)
\\
\nu_\eta\circ\zeta_\eta(t)
\\
n_\eta\circ\zeta_\eta(t)
\end{array}\right)= \zeta_\eta'(t)\left(\begin{array}{ccc}
0 & \kappa_{\Gamma_\eta}\circ\zeta_\eta(t)\ \ \  & \kappa_\eta(t)
\\
-\kappa_{\Gamma_\eta}\circ\zeta_\eta(t)\ \ \  & 0 & 0
\\
-\kappa_\eta(t) & 0 & 0
\end{array}\right)\left(\begin{array}{c}
\tau_\eta\circ\zeta_\eta(t)
\\
\nu_\eta\circ\zeta_\eta(t)
\\
n_\eta\circ\zeta_\eta(t)
\end{array}\right),
\end{equation}
with the initial conditions $\bigl(\tau_\eta\circ\zeta_\eta(0),\nu_\eta\circ\zeta_\eta(0),n_\eta\circ\zeta_\eta(0)\bigr)=\bigl(\tau(0),\nu(0),n(0)\bigr)$. Due to (\ref{6.137}) and the properties of $\zeta_\eta$ the coefficients on the right hand side of (\ref{6.141}) converge to the coefficients of system (\ref{2.108}) in $L^\infty$ norm. This implies  
\begin{equation}\label{144}
 \begin{gathered}
\big\|(\tau_\eta\circ\zeta_\eta,\nu_\eta\circ\zeta_\eta,n_\eta\circ\zeta_\eta)-(\tau,\nu,n)\big\|_{L^\infty[0,l]}\to 0,
\\
\|\gamma_\eta\circ\zeta_\eta -\gamma\|_{L^\infty[0,l]}\to 0,
\\
\big\|\mathcal U_\eta-u\big\|_{L^\infty(\omega[\Gamma])}\to 0.
 \end{gathered}
\end{equation}

Analogously we have 
\begin{equation*}
 \begin{gathered}
  \nabla' \mathcal U_\eta = \big(T_{\eta,1} \tau_\eta - T_{\eta,2} \nu_\eta,\,T_{\eta,2} \tau_\eta + T_{\eta,1} \nu_\eta\big),
  \\
  \nabla'^2  \mathcal U_\eta = \frac{\kappa_\eta\circ\zeta^{-1}_\eta}{1-\hat s \kappa_{\Gamma_\eta}}\left(
 \begin{array}{cc}
  T_{\eta,1}^2\, n_\eta & T_{\eta,1} T_{\eta,2}\, n_\eta
  \\
  T_{\eta,1} T_{\eta,2}\, n_\eta &  T_{\eta,2}^2\, n_\eta
 \end{array}
\right),
 \end{gathered}
\end{equation*}
where the terms on the right hand sides are functions of $\hat t\in [0,l_\e]$, {\it cf.} (\ref{6.139}). Let us substitute $\hat t = \zeta_\eta(t)$ and $\hat s = s - (\Gamma_\eta\circ\zeta_\eta(t)- \Gamma(t))\cdot N(t)$ in the above, so that $\Gamma_\eta(\hat t) + \hat s N_\eta(\hat t) = \Gamma(t) +  s N(t)$ everywhere in $\o[\Gamma]\setminus \o_\eta[\Gamma]$. We obtain, in particular, via (\ref{6.137}) and (\ref{6.139a}) that
\begin{equation*}
 \frac{\kappa_\eta\circ\zeta^{-1}_\eta(\hat t)}{1-\hat s \kappa_{\Gamma_\eta}(\hat t)}= 
 \left\{ \begin{array}{cc}
                 \dfrac{\kappa(t)}{1- s \kappa_{\Gamma}(t)} & \mbox{ on }\o[\Gamma]\setminus \o_\eta[\Gamma],
                 \\
                 \kappa(t) & \mbox{ on } \o_\eta[\Gamma].
                                                                                                     \end{array}\right.
\end{equation*}
Then from (\ref{135}), (\ref{6.135a}) and (\ref{144}) we derive the following convergence properties:
\begin{equation}\label{153}
 \begin{gathered}
 \big\|\nabla'\mathcal U_\eta-\nabla' u\big\|_{L^\infty(\omega[\Gamma])}\to 0,
 \\
 \big\|\nabla'^2\mathcal U_\eta-\nabla'^2 u\big\|_{L^2(\omega[\Gamma])}\to 0,
 \\
 \big\|\nabla'^2\mathcal U_\eta-\nabla'^2 u\big\|_{L^\infty(\o[\Gamma]\setminus \o_\eta[\Gamma])}\to 0,
 \\
{\rm II}_\eta={\rm II} \mbox{ on }\o[\Gamma]\setminus \o_\eta[\Gamma], \mbox{ as } \eta \to 0,
 \end{gathered}
\end{equation}
where ${\rm II}_\eta$ is the matrix of the second fundamental form of $\mathcal U_\eta$. 

The estimates (\ref{137}), (\ref{153}) imply the validity of (\ref{134}). Indeed, since the number of bodies and arms of $u$ is finite, we can carry out the above construction for every arm (on bodies the surface remains unchanged), and then assemble an $H^2(\omega)$ isometric surface starting from an arbitrary arm or body and attaching sequentially to a current piece the next one by applying a rigid motion, so that the resulting surface satisfies the properties analogous to (\ref{153}). Then choosing a sufficiently small $\eta$ and denoting the corresponding surface by $u_\delta$ and the union of all $S_\eta^\Gamma$ by $S_\delta$ yields the estimate (\ref{134}). Finally, we split each arm of $u_\delta$ into smaller arms $\omega[\tilde\Gamma_r]$ corresponding to the cylindrical parts with rational directions according to the construction and their complement.

We conclude the proof with an argument for additional regularity of $u_\delta$, which is claimed in the statement. On each arm we can approximate the corresponding curvature $\kappa$ in the $L^2$-norm by an infinitely smooth function. We then reconstruct the pieces of the new surface using the equations (\ref{2.108}), (\ref{6.136}) and $\gamma' = \tau$ with appropriate initial conditions, and attach them together in a fashion similar to the above. Taking sufficiently good approximations of $\kappa$ we ensure the validity of the estimate (\ref{134}) for the new surface. 

The following scaling argument allows us to obtain the boundedness of ${\rm II}_\delta$ in $\omega$. Due to the smoothness of $\kappa$ and (\ref{147}) the matrix of the second fundamental form ${\rm II}_\delta$ is bounded on any compact subset of $\omega$. Without loss of generality we can assume that $\omega$ contains the origin. Since $\omega$ is convex by the assumption, it is a star-shaped domain, {\it i.e.} $\omega_{\eta}:=(1-\eta)\omega\subset \omega$, where we use $\eta$ again to denote a small positive parameter. Let us define the scaled surface: $u_{\delta,\eta}(x'):=(1-\eta)^{-1}u_\delta\bigl((1-\eta)x'\bigr)$, $x'\in\omega$. It is easy to see that $u_{\delta,\eta}$ is an isometric surface (this is why we need to use the coefficient $(1-\eta)^{-1}$ in front of $u_\delta$). In addition, one has ${\rm II}_{\delta,\eta}(x')=(1-\eta){\rm II}_\delta\bigl((1-\eta)x'\bigr)$ and ${\rm II}_{\delta,\eta}\in L^\infty(\omega)$. Therefore, for any quadratic form $Q$, in particular for $Q_{\rm hom}^{\rm sc}$, we obtain
\begin{equation*}
 \int\limits_\omega Q\bigl({\rm II}_{\delta,\eta}(x')\bigr)dx' =(1-\eta)^2 \int\limits_\omega Q\big({\rm II}_\delta((1-\eta)x')\big)dx' =  (1-\eta) \int\limits_{\omega_\eta} Q\bigl({\rm II}_{\delta}(x')\bigr)dx'.
\end{equation*}
Hence, choosing small enough $\eta$ and re-denoting $u_{\delta,\eta}$  by $u_\delta$ we arrive at (\ref{134}).
\end{proof}

\section{Limit functional and $\Gamma$-convergence in the ``supercritical'' regime $h\ll \e^2$}\label{extremeenergy}

As follows from the compactness result of Section \ref{compactness} and the isometry constraint of Section \ref{s5}, in the case when $h\ll \e^2$ we should define the homogenised energy functional as
\begin{equation}\label{63}
 E_{\rm hom}^{\rm sc}(u):= \frac{1}{12}\int\limits_\omega Q_{\rm hom}^{\rm sc}\bigl({\rm II}\bigr) dx',
\end{equation}
with the elastic stored-energy function
\begin{equation}\label{2.103}
 Q_{\rm hom}^{\rm sc}({\rm II}) := \min \int\limits_Y Q_2\bigl(y, {\rm II} + \nabla_y^2 \psi\bigr) dy,\\\  \psi\in H^2_{\rm per}(Y),\\\ \det\bigl({\rm II} +\nabla_y^2 \psi\bigr) = 0.
\end{equation}

Zero determinant condition is rather restrictive in the case of periodic functions. Next we state a theorem describing the structure of periodic functions $\psi$ subject to this condition, first recalling that ${\rm II}$ can be represented in the form (\ref{147}). The coefficient $-\kappa(1-s \kappa_\Gamma)^{-1}$ is one of the two principal curvatures of $u$, the other principal curvature is always zero.

\begin{theorem}\label{th2.8}
Let $\rm II$ be a non-zero matrix of the form (\ref{147}) and let $\psi \in H^2_{\rm per}(Y)$ be such that 
\begin{equation}\label{2.99}
 \det\bigl({\rm II} +\nabla_y^2 \psi\bigr) = 0.
\end{equation}
Then the derivative of $\psi$ in the direction $(-T_2, \,T_1)^\top$ is zero, {\it i.e.} one has
\begin{equation*}
 \psi(y_1,y_2) = \psi_T(T_1 y_1+T_2 y_2)
\end{equation*}
for some $\psi_T$. In particular we have $\psi_T = \psi_T(y_1)$ or $\psi_T=\psi_T(y_2)$ and $\psi_T$ is $1$-periodic for $T=(1,0)^\top$ and $T=(0,1)^\top$ respectively. For other rational $T$ such that $T_1/T_2=p/q\in\mathbb Q$, where $p/q$ is an irreducible fraction, the function $\psi_T$ is periodic with period $P=\left|T_1/p\right|=
\left|T_2/q\right|$. For irrational vectors $T$ the function $\psi_T$ (and, hence the function $\psi$) is constant.
\end{theorem}
\begin{proof}
The proof is based on a result by Pakzad in \cite{Pakzad}. Let us denote 
\begin{equation*}
v(y):= \frac{1}{2}y_1^2 {\rm II}_{11} + y_1y_2 {\rm II}_{12} + \frac{1}{2} y_2^2 {\rm II}_{22} + \psi,
\end{equation*}
so that $\det (\nabla_y^2 v) = 0$ is equivalent to equation (\ref{2.99}). Then $\nabla v$ is continuous and for any regular convex domain $D$ in $\mathbb R^2$ and any $y\in D$ there exist a segment of line passing through $y$ and connecting two points on the boundary of $D$ such that $\nabla_y v = {\rm II} y + \nabla_y \psi$ is constant along this segment. Since the periodicity and continuity (by the Sobolev Embedding Theorem) of $\psi$ imply its boundedness, it follows that the segment along which $\nabla_y v$ is constant has to be parallel to the vector $(-T_2, \,T_1)^\top$: along this direction ${\rm II} y$ is constant. Hence $\nabla_y \psi$ is constant in the direction $(-T_2, \,T_1)^\top$, and the derivative of $\psi$ in this direction has to be zero due to the boundedness of $\psi$. The rest of the statement of the theorem easily follows from the periodicity of $\psi$.
\end{proof}

The above theorem implies that at the points $x'$ where the principal directions of the isometric surface $u$ are irrational we have $Q_{\rm hom}^{\rm sc} ({\rm II})= \int_Y Q_2(y, {\rm II}) dy$, while on the rational directions the corrector $\psi$ is non-trivial and in general one has $Q_{\rm hom}^{\rm sc} ({\rm II})< \int_Y Q_2(y, {\rm II}) dy$. This means that in general $Q_{\rm hom}^{\rm sc} ({\rm II})$ is an everywhere discontinuous function of the direction $T$. We can write $\psi$ in the form
\begin{equation}\label{156}
 \psi(y_1,y_2) = - \frac{\kappa}{1-s \kappa_\Gamma} \psi_T (T_1 y_1+T_2 y_2),
\end{equation}
where $\psi_T(t)$ is the solution of the minimisation problem 
\begin{equation}\label{157}
\min \int\limits_Y Q_2\big(y, \widetilde{\rm II}(1 + \psi_T''(T_1 y_1+T_2 y_2))\big) dy,\\\ \psi_T(t)\in H^2_{\rm per}(0,P),
\end{equation}
with $\widetilde{\rm II}$ defined by (\ref{53}). Notice that the Sobolev embedding theorem implies that $\psi_T(t)$ is continuously differentiable.

\begin{theorem}[$\Gamma$-convergence]
\label{th5.6} 
Suppose that $h\ll \e^2$, {\it i.e.} $h\e^{-2}=o(1)$ as $h\to 0$. Then the rescaled sequence of functionals $h^{-2} E_h$ $\Gamma$-converges to the limit functional $E_{\rm hom}^{\rm sc}$ in the sense that
\begin{enumerate}
 \item {\bf Lower bound.} For every bounded bending energy sequence $u_h\in H^1(\Omega)$ such that $\nabla' u_h$ converges to $\nabla' u$ strongly in $L^2(\Omega)$, $u\in H^2_{\rm iso}(\omega)$, the following inequality holds:
 \begin{equation*}
  \liminf_{h\to 0} h^{-2} E_h(u_h) \geq E_{\rm hom}^{\rm sc}(u).
 \end{equation*}
\item {\bf Recovery of the lower bound.} For every $u\in H^2_{\rm iso}(\omega)$ there exists a sequence $u_h^{\rm rec}\in H^1(\Omega)$ such that $\nabla' u_h^{\rm rec}$ converges to $\nabla' u$ strongly in $L^2(\Omega)$ and
 \begin{equation*}
  \lim_{h\to 0} h^{-2} E_h(u_h^{\rm rec}) = E_{\rm hom}^{\rm sc}(u).
 \end{equation*}
\end{enumerate}
\end{theorem}

\begin{proof} {\bf Lower bound.} The proof follows the derivation of the lower bound in the case when $\varepsilon\gg h\gg\varepsilon^2,$ up to the last inequality in 
(\ref{commonest}), whose right-hand side can have $Q_{\rm hom}^{\rm m}$ replaced by $Q_{\rm hom}^{\rm sc}$ given by (\ref{2.103}), thanks to Theorem \ref{isometryconstraint}.

{\bf Recovery sequence.} A natural method of approximating the deformation of an elastic plate is to first write an expression $u=u(x')$ for the mid-surface of the deformed plate, which takes into account the microscopic oscillations in the definition of the homogenised functional 
(represented by the function $\psi$), and then ``add thickness'' to this surface. One possible approach to deriving the mid-surface 
expression is via an asymptotic expansion, with $u(x')$ as the leading order term followed by a series of correctors in sequential powers 
of $\e$. In the case $h\ll \e^2$, the first corrector involving the function $\psi$ is shown to satisfy equation (\ref{2.99}), see 
Section \ref{asympsection}, which is the solvability condition for the equation at the next order in $\varepsilon.$ 
However, higher-order equations all involve two unknown functions and their solution is not straightforward. 
As was already mentioned in the previous section, in our construction we adopt a different approach based on the properties of isometric immersion, which we describe next. The details of the formal asymptotic argument are given in Section \ref{asympsection}.

As was shown in Theorem \ref{th2.8}, the corrector $\psi$ is constant whenever $T$ is irrational. This implies that if the measure of the set on which $T$ is rational is zero, then the limit elastic energy is simply
\begin{equation*}
\frac{1}{12}\int\limits_{\omega} Q_{\rm hom}^{\rm sc}({\rm II}) dx' = \frac{1}{12}\int\limits_{\omega\times Y} Q_2(y,{\rm II}) dy dx'.
\end{equation*}
Thus, the only case that remains to be considered is when the above mentioned set has non-zero measure. 

Let $u\in H^2_{\rm iso}$. By Theorem \ref{th6.6} we can choose a sequence of regularised surfaces $u_{\delta_j}$, $\delta_j\to 0$, such that 
 \begin{equation}\label{162a}
 \begin{gathered}
 \bigl|E_{\rm hom}^{\rm sc}(u) -  E_{\rm hom}^{\rm sc}(u_{\delta_j})\bigr|\leq \frac{1}{j},
 \\
 \int\limits_{S_{\delta_j}}\int\limits_Y Q_2(y, {\rm II_{\delta_j}}) dx'dy \leq \frac{1}{j}.
 \end{gathered}
\end{equation}
We first construct an approximating sequence of plate deformations for $u_{\delta_j}$ for each $j$ and then apply a diagonalisation procedure to obtain a recovery sequence for $u$. Notice that all objects below depend on the parameter $\delta_j$, however we drop it from the notation in most of the cases. Let us consider an arm $\o[\Gamma_r]$ of $u_{\delta_j}$ corresponding to a rational direction $T$. On this arm the leading curve $\Gamma_r$ is a segment of a straight line, $T$ is constant, $\kappa_{\Gamma_r}=0$, the principal curvature is $-\kappa(t)$ (see (\ref{147})), and the vector $\nu$ of the Darboux frame $(\tau,\nu,n)$ of $u_{\delta_j}\circ \Gamma_r$ is constant. Let $\psi = -\kappa(t)\psi_T(T_1 y_1 + T_2 y_2)$ be the minimiser of (\ref{2.103}), {\it cf.} also (\ref{156}). We have
\begin{equation*}
\begin{gathered}
  {\rm II} = -\kappa\left(
 \begin{array}{cc}
  T_1^2 & T_1 T_2 
  \\
  T_1 T_2  &  T_2^2 	
 \end{array}
\right),
\\
Q_{\rm hom}^{\rm sc}({\rm II}) = \min \int\limits_Y Q_2\bigl(y, {\rm II} - \kappa\nabla_y^2 \psi_T(T_1 y_1 + T_2 y_2)\bigr) dy\qquad\qquad\qquad
\\
\qquad\qquad\qquad=\min \int\limits_Y Q_2\big(y, {\rm II}(1 + \psi_T''(T_1 y_1+T_2 y_2))\big) dy.
\end{gathered}
\end{equation*}

 Let $\psi_{T,j}$ be infinitely smooth functions approximating $\psi_T$ in $L^2(0,l)$. We will make a more precise choice of these functions later. In order to incorporate the corrector $\psi_{T,j}$ in the equation of the approximating surface we need to understand the relation between the coordinates $(t,s)$, $x'$ and the underlying  $\e Y$-periodic lattice. From $x'=\Gamma(t)+sN=\Gamma(0)+tT+sN$ we derive $t=T\cdot (x'-\Gamma(0))$. Applying the unfolding operator to $\psi_{T,j}(\e^{-1}t) =\psi_{T,j}(\e^{-1} T\cdot (x'-\Gamma(0)))$ we get 
  \begin{equation*}
 \mathcal{T}_\e (\psi_{T,j}(\e^{-1}t))=\psi_{T,j}(T\cdot [\e^{-1}x'-\e^{-1}\Gamma(0)] + T\cdot y)
 \end{equation*}
 with $y=\{\e^{-1}x'-\e^{-1}\Gamma(0)\}$, where $[\cdot]$ and $\{\cdot\}$ denote the integer and the fractional parts. Hence in order to align the argument $\e^{-1}t$ with the lattice we need to shift it by adding $t_\e^* = T\cdot \{\e^{-1}\Gamma(0)\}$: 
 \begin{equation*}
  \mathcal{T}_\e (\psi_{T,j}(\e^{-1}t+t_\e^*))=\psi_{T,j}(T\cdot ([\e^{-1}x']-[\e^{-1}\Gamma(0)]) + T\cdot y)=\psi_{T,j}(T\cdot y)
 \end{equation*}
 with $y=\{\e^{-1}x'\}$. On the domain $[0,l]$ of $\Gamma_r$ we define a corrected tangent vector
\begin{equation}\label{163}
 \tau_\e(t):= \tau(t)\sqrt{1-(\e\kappa(t)\psi_{T,j}'(\e^{-1}t+t_\e^*))^2} + \e \kappa(t) \psi_{T,j}'(\e^{-1}t+t_\e^*)n(t),
\end{equation}
so that $\bigl\vert \tau_\varepsilon(t)\vert=1$. For small values of $\e,$ the square root in the above is well defined. The approximating leading curve $\gamma_\e$ is given by integration of $\tau_\e$, namely
\begin{equation*}
\gamma_\e(t):= \int\limits_{0}^t \tau_\e,\ \ \ \gamma_\e(0) = \gamma(0).
\end{equation*}
Now we set
\begin{equation*}
 u_\e\bigl(\Gamma(t)+sN\bigr):= \gamma_\e(t)+s\nu,\ \ {\rm whenever}\ \ \Gamma(t)+sN\in{\omega[\Gamma_r]}.
\end{equation*}

Let us estimate the difference between $u$ and $u_\e$. First, it is clear that 
\begin{equation}\label{165}
\begin{gathered}
  \tau_\e=\tau + O(\e), \,\,\,  n_\e=n + O(\e),\,\,\,\gamma_\e=\gamma + O(\e), \mbox{ as } \e\to 0,
\end{gathered}
\end{equation}
where $n_\e:=\tau_\e\wedge \nu$. Here and until the end of this section $o(\e),\,O(\e)$ are understood in terms of the $L^\infty$-norm. Hence
\begin{equation*}
 \|u_\e-u_{\delta_j}\|_{L^\infty(\o[\Gamma_r])}\to 0, \mbox{ as } \e\to 0.
\end{equation*}
Second, we have ({\it cf.} (\ref{145}))
\begin{equation*}
 \begin{gathered}
  \nabla' u_\e = \big(T_1 \tau_\e - T_2 \nu\vert\,T_2 \tau_\e + T_1 \nu\big).
 \end{gathered}
\end{equation*}
It follows immediately that 
\begin{equation}\label{169}
\begin{gathered}
\|\nabla' u_\e-\nabla' u_{\delta_j}\|_{L^\infty(\o[\Gamma_r])}\to 0.
\end{gathered}
\end{equation}

We carry out this procedure for all rational arms $\o[\Gamma_r]$ leaving the surface unchanged on the rest of $\o$. We then assemble a new surface denoting it by $u_{\e}$ in the same way it was done in the proof of Theorem \ref{th6.6} using rigid motions. It is clear from the construction that 
\begin{equation}\label{174a}
 u_{\e} \to u_{\delta_j}\mbox{ in } W^{1,\infty}(\o)\mbox{ and in }H^1(\omega).
\end{equation}

Now we ``add thickness'' to the plate and define the approximating deformation as
\begin{equation}\label{5.120}
 u_{\delta_j,h}:= u_\e(x') + h x_3 n_\e(x') + h^2 \frac{x_3^2}{2} d_j(x',\e^{-1}x'),
\end{equation}
where $d_j(x',y)$ is some function from $C^\infty(\overline{\o}\times Y)$ periodic with respect to $y$, which will be chosen later and the normal $n_\e$ is defined by $n_\e=  \partial_1 u_\e\wedge \partial_2 u_\e$ on non-rational arms and bodies.
The gradient of  $u_{\delta_j,h}$ is given by
\begin{equation*}
 \nabla_h u_{\delta_j,h}= (\nabla' u_\e \vert \,n_\e)  + h x_3(\nabla' n_\e\vert\, d_j)+  h^2 \frac{x_3^2}{2} (\nabla' d_j+ \e^{-1} \nabla'_y d_j\vert\, 0).
 \end{equation*}
Clearly, $(\nabla' u_\e, \,n_\e)\in \SO 3,$ thus we have
\begin{equation*}
 h^{-1}\Bigl((\nabla' u_\e\vert \,n_\e)^\top \nabla_h u_{\delta_j,h}-{\mathcal I}_3\Bigr)=  x_3(\nabla' u_\e\vert \,n_\e)^\top (\nabla' n_\e\vert\, d_j)+  o(\e).
 \end{equation*}
On each arm $\o(\Gamma)$ which does not belong to the set of rational arms $\o(\Gamma_r)$ the surface $u_\e$ is a rigid motion of $u$. Such a transformation preserves the matrix of the second fundamental form, {\it i.e.} for the matrix of the second fundamental form of $u_\e$ we have ${\rm II}_\e=(\nabla' u_\e)^\top \nabla' n_\e = {\rm II}$ on $\o(\Gamma)$.

Let us consider a rational arm $\o(\Gamma_r)$. Notice that for $\kappa_\e$, which is given via $\tau_\e' = \kappa_\e n_\e$, after a simple calculation using (\ref{163}) and (\ref{165}) one has
\begin{equation}\label{168}
 \kappa_\e = \kappa+ \kappa\psi_{T,j}''(\e^{-1}t) + O(\e).
\end{equation} Since $n_\e'=-\kappa_\e \tau_\e$ we have
\begin{equation*}
\partial_i n_\e= -\kappa_\e \tau_\e T_i,\,\, i=1,2.
\end{equation*}
Then from (\ref{165}), (\ref{169}) and (\ref{168}) we obtain
\begin{equation*}
(\nabla' u_\e)^\top \nabla' n_\e= -(\kappa+\kappa\psi_{T,j}'')\widetilde {\rm II}+ O(\e).
\end{equation*}
We conclude that
\begin{equation}\label{175}
 h^{-1}\bigl((\nabla' u_\e, \,n_\e)^\top \nabla_h u_{\delta_j,h}-{\mathcal I}_3\bigr)= \left\{ \begin{array}{cc}
  x_3\big({\rm II} \vert\, \tilde d_j \big)+  o(\e) & \mbox{ if } x'\in \omega_{ir}\cap S_\delta,
  \\
    x_3\big(-(\kappa+\kappa\psi_{T,j}'')\widetilde {\rm II} \vert\, \tilde d_j \big)+  O(\e) & \mbox{ if } x'\in \omega_r,
 \end{array}
\right.
\end{equation}
where $\tilde d_j:=(\nabla' u , \,n)^\top d_j$, and $\o = \o_r\cup\o_{ir}\cup S_\delta$ is the natural decomposition of $\o$ into the union of rational arms $\o_r=\bigcup_{\Gamma_r} \o[\Gamma_r]$, the residual set of rational segments $S_\delta$ and the rest of the set $\omega_{ir}$. We understand the notation $\psi_{T,j}$ as a family of functions whose members are related to rational arms $\o[\Gamma_r]$ via the corresponding directions $T$.

Now we are ready to define a recovery sequence and prove the convergence of the elastic energy. We use the frame indifference of $W$, its Taylor expansion near 
${\mathcal I}_3$, the unfolding operator and (\ref{175}) to get
\begin{equation*}
\begin{gathered}
 h^{-2}E_h(u_{\delta_j,h})=h^{-2}\int\limits_{\Omega} W(\e^{-1}x', \nabla u_{\delta_j,h}) dx = h^{-2}\int\limits_{\Omega} W\bigl(\e^{-1}x',(\nabla' u_\e\vert \,n_\e)^\top\nabla_h u_{\delta_j,h}\bigr) dx =
 \\
 =\frac{1}{12} \int\limits_{ \omega_{ir}\cap S_\delta}\int\limits_Y Q_3\bigl(y, \mathcal T_\e({\rm II}\vert\, \tilde d_j)\bigr) dx'dy + \frac{1}{12} \int\limits_{\omega_r}\int\limits_Y Q_3\bigl(y, \mathcal T_\e(-(\kappa+\kappa\psi_{T,j}'')\widetilde {\rm II}\vert\,\tilde d_j)\bigr) dx'dy+ \int\limits_{\Omega} O(\e)dx
 \\
 +\frac{1}{12} \int\limits_{\Lambda_\e}Q_3\big(\e^{-1}x', h^{-1}((\nabla' u_\e, \,n_\e)^\top \nabla_h u_{\delta_j,h}-{\mathcal I}_3)\big) dx'.
 \end{gathered}
\end{equation*}
Notice that we can use the Taylor expansion of $W$ due to the fact that ${\rm II} \in L^\infty(\omega)$ by Theorem \ref{th6.6}. Passing to the limit in the above as $h\to 0$ and taking into account the properties of two-scale convergence we arrive at
\begin{equation*}
\begin{gathered}
 \lim_{h\to 0}h^{-2}E_h(u_{\delta_j,h})
 =\frac{1}{12} \int\limits_{ \omega_{ir}\cap S_\delta}\int\limits_Y Q_3\bigl(y, ({\rm II}\vert\, \tilde d_j)\bigr)dx'dy \qquad\qquad\qquad\qquad\qquad\qquad
 \\
 \qquad\qquad\qquad\qquad\qquad\qquad+ \frac{1}{12} \int\limits_{\omega_r}\int\limits_Y Q_3\bigl(y, (-(\kappa+\kappa\psi_{T,j}''(T\cdot y))\widetilde {\rm II}\vert\,\tilde d_j)\bigr) dx'dy.
 \end{gathered}
\end{equation*}

Choosing $\psi_{T,j}$ and $d_j$, {\it i.e.} close enough to the minimisers of the problems (\ref{157}) and (\ref{58}) with appropriate arguments we can make the term 
\begin{equation*}
\begin{gathered}
\frac{1}{12} \int\limits_{ \omega_{ir}}\int\limits_Y Q_3\bigl(y, ({\rm II}\vert\,\tilde d_j)\bigr) dx'dy+ \frac{1}{12} \int\limits_{\omega_r}\int\limits_Y Q_3\bigl(y, (-(\kappa+\kappa\psi_{T,j}''(T\cdot y))\widetilde {\rm II}\vert\,\tilde d_j)\bigr) dx'dy
 \end{gathered}
\end{equation*}
arbitrary close to 
\begin{equation*}
\begin{gathered}
\frac{1}{12} \int\limits_{ \omega_{ir}\cup \omega_r} Q_{\rm hom}^{\rm sc} ({\rm II}) dx'.
 \end{gathered}
\end{equation*}
For the term corresponding to the residual set $S_{\d_j}$ we have, in general, 
\begin{equation*}
 \int\limits_{S_{\delta_j}\times Y} Q_3\bigl(y, ({\rm II}\vert\, \tilde d_j)\bigr) dx'dy>\int\limits_{S_{\delta_j}} Q_{\rm hom}^{\rm sc} ({\rm II}) dx',
\end{equation*}
however thanks to the estimate (\ref{162a}) it can be controlled. Thus, for each $j$ we can choose $\psi_{T,j}$ and $d_j$ so that
\begin{equation*}
\begin{gathered}
 \bigl|\lim_{h\to 0}h^{-2}E_h(u_{\delta_j,h})-E_{\rm hom}^{\rm sc}(u_{\delta_j})\bigr | \leq \frac{2}{j}.
 \end{gathered}
\end{equation*}
The latter, together with (\ref{162a}), implies that we can choose a subsequence $h_j\to 0$ so that 
\begin{equation*}
\begin{gathered}
\bigl |h^{-2}E_h(u_{\delta_j,h_j})-E_{\rm hom}^{\rm sc}(u)\bigr | \leq \frac{4}{j}.
 \end{gathered}
\end{equation*}
Passing to the limit as $j\to \infty$ we get the convergence of energies. Thus
\begin{equation*}
 u_h^{\rm rec} := u_{\delta_j,h_j}
\end{equation*}
is a recovery sequence. The convergence of the gradients follows from (\ref{134}) (\ref{174a}) and the construction procedure.
\end{proof}

\appendix
\section{Appendix: Two-scale convergence and its properties}\label{twoscaleappendix}
 Let $\omega\subset \mathbb R^d$ be a bounded domain and $Y:=[0,1)^d$ (in the present paper $d=2$). We denote by $[x]$ the integer part of $x$, and by $\{x\}$ its fractional part, {\it i.e.} $[x]\in \mathbb Z^d$, $\{x\}:=x-[x] \in Y$. The definition uses the unfolding operator which maps functions of one variable $x$ to functions of two variables $x$ and $y$, where the ``fast'' variable $y$ is responsible for the function behaviour on $\e$-scale: 
\begin{equation*}
 \mathcal T_\e (v)(x,y):= \left\{\begin{array}{ll}
                            v\left(\e \left[\frac{x}{\e}\right] + \e y\right) & \mbox{ if } x\in \e(\xi+ Y) \subset \omega \mbox{ for } \xi \in \mathbb Z^d,
                            \\
                            0 & \mbox{\ otherwise,}
                           \end{array}
 \right. 
\end{equation*}
for a measurable on $\omega$ function $v(x)$. Notice that $\mathcal T_\e (v)(x,y)$ is set to be zero on the cells intersecting with the boundary of $\omega$. This implies, in particular, that if $v\in H^1(\omega)$ then $\mathcal T_\e (v)\in L^2(\omega; H^1(Y))$. One has 
\begin{equation*}
 \int\limits_\omega v dx = \int\limits_{\omega\times Y}\mathcal T_\e (v) dy dx+  \int\limits_{\Lambda_\e} v dx,
\end{equation*}
where $\Lambda_\e:= \{x\in \e(\xi+ Y)\cap\omega: \, \e(\xi+ Y)\cap\partial\omega\neq \varnothing,\,\xi \in \mathbb Z^d\}$.

\begin{definition}
 We say that a bounded in $L^p(\omega)$ sequence $v_\e$ converges weakly (strongly) two-scale to $v\in L^p(\omega\times Y)$ as $\e\to 0$ and denote this by $v_\e \wtto v$ ($v_\e \stto v$) if $\mathcal T_\e (v_\e)$ converges to $v$ weakly (strongly) in $L^p(\omega\times Y)$.
\end{definition}
We refer the reader to \cite{CDG} for a comprehensive analysis of the properties of the unfolding operator $\mathcal T_\e$, which, in particular, links the convergence of sequences in $L^2(\omega)$ to convergence of their unfoldings in $L^2(\omega\times Y)$.

We also need a formal definition of two-scale convergence when not all variables have their fast scale counterparts, which is just a variation of the standard definition.
\begin{definition}
 Using the notation accepted in this paper ({\it i.e.} $\Omega= \omega\times I$, $\omega\subset \mathbb{R}^2$, $Y=[0,1)^2$, $x=(x',x_3)$), we say that a bounded in $L^p(\Omega)$ sequence $v_\e$ converges weakly (strongly) two-scale to $v\in L^p(\Omega\times Y)$ as $\e\to 0$ and denote this by $v_\e \wtto v$ ($v_\e \stto v$) if $\mathcal T_\e (v_\e)$ (where $\mathcal T_\e$ acts only with respect to $x'$) converges to $v$ weakly (strongly) in $L^p(\Omega\times Y)$.
\end{definition}


\section{Appendix: Formal asymptotic expansion for a recovery sequence}
\label{asympsection}

In this appendix we present a formal asymptotic expansion for a recovery sequence which leads to the zero-determinant condition (\ref{2.71}).

Let $u(x')$ be an arbitrary isometric surface, $u\in H^2_{\rm iso}(\omega)$, and $\psi(x',y)$ be a sufficiently smooth scalar function.
As a starting point we consider the surface 
\begin{equation}\label{apr2}
u_\e^{{\rm ap},2}:= u(x')-\e^2 \psi(x',\e^{-1}x') n(x')
\end{equation}
(similar to the case $\e\gg h\gg \e^2$). A simple calculation ({\it cf.} (\ref{4.87})) shows that $(\nabla' \tilde u_\e^{{\rm ap},2})^\top \nabla' \tilde u_\e^{{\rm ap},2} = {\mathcal I}_2+O(\e^2)$ at least formally. Ideally we need to add a sufficient number of corrector terms to the expression in (\ref{apr2}) to make the error to be of order $o(h)$. Let us consider a third order approximation, 
\begin{equation*}
u_\e^{{\rm ap},3}:= u(x')-\e^2 \psi(x',\e^{-1}x') n(x') + \e^3 \varphi(x',\e^{-1}x'),
\end{equation*}
where $\varphi$ is a sufficiently smooth vector-valued function. Taking the gradient we get
\begin{equation*}
\begin{gathered}
 \nabla' u_\e^{{\rm ap},3}= \nabla' u-\e \nabla_y \psi\, n -\e^2\big(\nabla' \psi\, n +  \psi \nabla'n + \nabla_y \varphi\big)  + \e^3 \nabla' \varphi.
\end{gathered}
\end{equation*}
Noticing that 
\begin{equation*}
\begin{gathered}
 \partial_1 n =  {\rm II}_{11} R_1 + {\rm II}_{12} R_2, \qquad  \partial_2 n =  {\rm II}_{12} R_1 + {\rm II}_{22} R_2
\end{gathered}
\end{equation*}
we have 
\begin{equation*}
\begin{gathered}
 |\partial_1 u_\e^{{\rm ap},3}|^2 =  1 + \e^2 \big( (\partial_{y_1} \psi)^2 - 2 \psi\, {\rm II}_{11} + 2 R_1\cdot \partial_{y_1}\varphi \big)  + O(\e^3),
\\
 |\partial_2 u_\e^{{\rm ap},3}|^2 =  1 + \e^2 \big( (\partial_{y_2} \psi)^2 - 2 \psi\, {\rm II}_{22} + 2 R_2\cdot \partial_{y_2}\varphi\big)  + O(\e^3),
\\
 \partial_1 u_\e^{{\rm ap},3}\cdot \partial_2 u_\e^{{\rm ap},3} =  \e^2 \big( \partial_{y_1} \psi \,\partial_{y_2} \psi - 2 \psi\, {\rm II}_{12} +  R_1\cdot \partial_{y_2}\varphi + R_2\cdot \partial_{y_1}\varphi \big)  + O(\e^3).
\end{gathered}
\end{equation*}
Let us equate the second-order terms to zero.
\begin{equation*}
\begin{gathered}
(\partial_{y_1} \psi)^2 - 2 \psi\, {\rm II}_{11} + 2 R_1\cdot \partial_{y_1}\varphi = 0,
\\
(\partial_{y_2} \psi)^2 - 2 \psi\, {\rm II}_{22} + 2 R_2\cdot \partial_{y_2}\varphi=0,
\\
\partial_{y_1} \psi \,\partial_{y_2} \psi - 2 \psi\, {\rm II}_{12} +  R_1\cdot \partial_{y_2}\varphi + R_2\cdot \partial_{y_1}\varphi=0.
\end{gathered}
\end{equation*}
Clearly, the function $\psi$ has to satisfy certain solvability condition. Indeed, differentiating the above equations twice with respect to $y$, we get
\begin{equation*}
\begin{gathered}
 R_1\cdot \partial^3_{y_1 y_2 y_2}\varphi =  - (\partial^2_{y_1 y_2} \psi)^2 - \partial_{y_1} \psi\,\partial^3_{y_1 y_2 y_2} \psi +  \partial^2_{y_2 y_2} \psi\, {\rm II}_{11},
\\
 R_2\cdot \partial^3_{y_1 y_1 y_2}\varphi =  - (\partial^2_{y_1 y_2} \psi)^2 - \partial_{y_2} \psi\,\partial^3_{y_1 y_1 y_2} \psi +  \partial^2_{y_1 y_1} \psi\, {\rm II}_{22},
 \\
\partial^2_{y_1 y_1} \psi \,\partial_{y_2 y_2} \psi + (\partial^2_{y_1 y_2} \psi)^2 + \partial_{y_1} \psi \,\partial^3_{y_1 y_2 y_2} \psi + \partial_{y_2} \psi \,\partial^3_{y_1 y_1 y_2} \psi- 2 \partial^2_{y_1 y_2} \psi\, {\rm II}_{12} \qquad\qquad\qquad
\\
\qquad\qquad\qquad\qquad\qquad\qquad\qquad\qquad +  R_1\cdot \partial^3_{y_1 y_2 y_2}\varphi + R_2\cdot \partial^3_{y_1 y_1 y_2}\varphi=0.
\end{gathered}
\end{equation*}
Eliminating the terms that contain derivatives of $\varphi$ and recalling that $\det {\rm II}=0$ yields 
\begin{equation*}
\begin{gathered}
\det ({\rm II} + \nabla_y^2 \psi) = 0.
\end{gathered}
\end{equation*}

\section*{Acknowledgements}

This work was carried out under the financial support of the Engineering and Physical Sciences Research Council (Grants EP/H028587/1 ``Rigorous derivation of moderate and high-contrast nonlinear composite plate theories'' and EP/L018802/1 ``Mathematical foundations of metamaterials: homogenisation, dissipation and operator theory''). We would like to thank Dr Shane Cooper for helpful remarks.


\begin{thebibliography}{9}

\bibitem{Allaire}  G. Allaire, 1992. Homogenisation and two-scale convergence, {\it SIAM J. Math. Analysis,} {\bf 23}(6) 1482--1518.    

\bibitem{CherCher} M. Cherdantsev, K. D. Cherednichenko, 2012. Two-scale $\Gamma$-convergence of integral functionals and its application to homogenisation of nonlinear high-contrast periodic composites, {\it Archive for Rational Mechanics and Analysis}, {\bf 204} 445--478.

\bibitem{CCN} Cherdantsev, M., Cherednichenko, K.D., Neukamm, S., 2013. Homogenisation in finite elasticity for composites with a high contrast in the vicinity of rigid-body motions, {\it Preprint arXiv:1303.1224}


\bibitem{Ciarlet} P. G. Ciarlet, 1988. {\it Mathematical Elasticity, Vol. I: Three-Dimensional Elasticity,} North-Holland.

\bibitem{CDG} D. Cioranescu, A. Damlamian, G. Griso, 2008. The periodic unfolding method in homogenization. {\it SIAM J. Math. Analysis}, {\bf 40}(4) 1585--1620.

\bibitem{DeGiorgiFranzoni} E. De Giorgi, T. Franzoni, 1975. Su un tipo di convergenza variazionale. {\it Atti Accad. Naz. Lincei Rend. Cl. Sci. Mat. Natur.}, {\bf 58} 842--850.

\bibitem{Dacorogna}
B. Dacorogna, 2000. {\it Direct Methods in Calculus of Variations}, Springer.

\bibitem{FJM2002} G. Friesecke, R. D. James, S. M\"{u}ller, 2002. A theorem on geometric rigidity and the derivation of nonlinear plate theory from three-dimensional Elasticity. {\it Comm. Pure Appl. Math.}, {\bf 55}(11) 1461--1506.

\bibitem{FJM_hierarchy} G. Friesecke, R. D. James, S. M\"{u}ller, 2006. A hierarchy of plate models derived from nonlinear elasticity by Gamma-convergence. {\it Arch. Rat. Mech. Analysis}, {\bf 180}(2) 183--236.

\bibitem{Griso} G. Griso, 2004. Error estimate and unfolding for periodic homogenization. {\it Asymptot. Anal.}, {\bf 40}(3-4) 269--286.
  
\bibitem{HLP}  G. H. Hardy, J. E. Littlewood, G. Polya, 1952. {\it Inequalities. Cambridge Mathematical Library} (2nd ed.), Cambridge University Press.

\bibitem{Hornung2008} P. Hornung, 2008. Approximating $W^{2,2}$ isometric immersions. {\it Comptes Rendus Mathematique}, {\bf 346}(3) 189--192. 

\bibitem{Hornung2011} P. Hornung, 2011. Approximation of flat $W^{2,2}$ isometric immersions by smooth ones. {\it Archive for Rational Mechanics and Analysis}, {\bf 199}(3) 1015--1067. 
  
\bibitem{HNV} P. Hornung, S. Neukamm, I. Velcic, 2014. Derivation of a homogenized nonlinear plate theory from 3d elasticity. {\it Calculus of Variations and Partial Differential Equations}, {\bf 51}(3-4) 677--699.
 
\bibitem{MuellerPakzad}  S. M\"{u}ller, M. R. Pakzad, Regularity properties of isometric immersions, 2005. {\it Mathematische Zeitschrift}, {\bf 251}(2) 313--331. 

 \bibitem{NeukammPhD2010} S. Neukamm, 2010. {\it Homogenization, linearization and dimension reduction in elasticity with variational methods}, Ph.D. Thesis, Munich.

\bibitem{NO} S. Neukamm, H. Olbermann, 2014. Homogenization of the nonlinear bending theory for plates, {\it Calculus of Variations and Partial Diff. Eq.}.

\bibitem{Nguetseng} G. Nguetseng, 1989. A general convergence result for a functional related to the theory of homogenization, {\it SIAM J. Math. Anal.}, {\bf 20} 608--629.

\bibitem{Pakzad} M. R. Pakzad, 2004. On the Sobolev space of isometric immersions, {\it J. Differential Geom.}, \textbf{66}(1) 47--69.

\bibitem{Pantz} O. Pantz, 2001. {\it Quelques probl\`{e}mes de mod\'{e}lisation en \'{e}lasticit\'{e} non lin\'{e}aire (Some modeling problems in nonlinear elasticity)}, These de l'Universite Pierre et Marie Curie.

\bibitem{Spivak} M. Spivak, 1979. {\it A Comprehensive Introduction to Differential Geometry, Vol. III,} Publish or Perish Inc.
  
\bibitem{Velcic} I. Velcic, 2013.  A note on the derivation of homogenized bending plate model, {\it Preprint arXiv:1212.2594.} 

\bibitem{JKO}  V. V. Zikov, S. M. Kozlov, O. A. Oleinik, 1994. {\it Homogenisation of Differential Operators and Integral Functionals,} Springer.

\bibitem{ZhikPast} V. V. Zhikov, S. E. Pastukhova, 2005. On operator estimates for some problems in homogenisation theory. {\it Russ. J. Math. Phys.}, {\bf 12}(4) 515--524.

\end{thebibliography}
\end{document}